\documentclass[11pt]{amsart}
\usepackage{float}
\usepackage[usenames]{color}
\usepackage{graphicx}
\usepackage{amssymb}
\usepackage{amscd}
\usepackage{makecell}
\theoremstyle{plain}
\newtheorem{thm}{Theorem}[section]
\newtheorem{lemma}[thm]{Lemma}

\newtheorem{prop}[thm]{Proposition}

\theoremstyle{definition}

\newtheorem{rmk}[thm]{Remark}
\newtheorem{example}[thm]{Example}

\def\dim{\mathop{\hbox {dim}}\nolimits}

\def\ker{\mathop{\hbox{Ker}}\nolimits}

\newcommand{\fra}{\mathfrak{a}}
\newcommand{\frb}{\mathfrak{b}}

\newcommand{\frg}{\mathfrak{g}}
\newcommand{\frh}{\mathfrak{h}}

\newcommand{\frk}{\mathfrak{k}}
\newcommand{\frl}{\mathfrak{l}}

\newcommand{\frp}{\mathfrak{p}}
\newcommand{\frq}{\mathfrak{q}}

\newcommand{\frt}{\mathfrak{t}}
\newcommand{\fru}{\mathfrak{u}}

\newcommand{\bbC}{\mathbb{C}}

\newcommand{\bbZ}{\mathbb{Z}}

\newcommand{\caC}{\mathcal{C}}

\begin{document}

\title{Dirac series of $E_{8(-24)}$}

\author{Yi-Hao Ding}
\address[Ding]{School of Mathematical Sciences, Soochow University, Suzhou 215006,
P.~R.~China}
\email{435025738@qq.com}

\author{Chao-Ping Dong}
\address[Dong]{School of Mathematical Sciences, Soochow University, Suzhou 215006,
P.~R.~China}
\email{chaopindong@163.com}

\author{Chengyu Du}
\address[Du]{School of Mathematical Sciences, Soochow University, Suzhou 215006,
P.~R.~China}
\email{cydu0973@suda.edu.cn}

\author{Yong-Zhi Luan}
\address[Luan]{School of Mathematics, Shandong University, Jinan 250100,
P.~R.~China}
\email{luanyongzhi@email.sdu.edu.cn}

\author{Liang Yang}
\address[Yang]{College of Mathematics, Sichuan University, Chengdu 610064,
P.~R.~China}
\email{malyang@scu.edu.cn}

\abstract{This paper classifies the Dirac series of $E_{8(-24)}$, the linear quaternionic real form of complex $E_8$. One tool for us is a further sharpening of the Helgason-Johnson bound in 1969. Our calculation continues to support Vogan's fundamental parallelepiped conjecture.}

\endabstract



\maketitle

%\tableofcontents

\section{Introduction}
This paper aims to carry on the classification of Dirac series (that is, irreducible unitary representations with non-zero Dirac cohomology) for real exceptional Lie groups. Our target group here is $E_{8(-24)}$, the linear real form of complex $E_8$ such that $G/K$ is of quaternionic type in the sense of Wolf \cite{Wo}. It is worth noting that Dirac series of complex $E_8$ has been classified by Barbasch and Wong recently \cite{BW}. It took us one year and a half to finish the current work. We believe that the split $E_8$ case is still very challenging.

Here and henceforth, we denote by $G$ a connected simple real reductive Lie group with finite center, by $\theta$ a Cartan involution of $G$, and by $K:=G^{\theta}$ a maximal compact subgroup of $G$. Then
$$
\frg_0=\frk_0+\frp_0
$$
is the Cartan decomposition on the Lie algebra level, with the group level counterpart being $G=K\exp(\frp_0)$. Take a maximal torus $T_f$ of $K$, and write its Lie algebra as $\frt_{f, 0}$. Let $\fra_{f, 0}=Z_{\frp_0}(\frt_{f, 0})$ and put $A_f=\exp(\fra_{f, 0})$. Then $H_f:=T_f A_f$ is a $\theta$-stable maximal compact Cartan subgroup of $G$. As usual, we will drop the subscripts to denote the complexified Lie algebras. We fix once for all a positive root system $\Delta^+(\frk, \frt_f)$ and denote the half-sum of the roots as $\rho_c$. We choose a positive root system $(\Delta^+)^{(0)}(\frg, \frt_f)$ containing $\Delta^+(\frk, \frt_f)$, and denote the half-sum of the roots by $\rho^{(0)}$. Correspondingly, we have Weyl groups $W(\frk, \frt_f)$ and $W(\frg, \frt_f)$. We fix an invariant non-degenerate bilinear form $B(\cdot, \cdot)$ on $\frg$. Its restrictions to $\frh_f$, $\frh_f^*$ etc will be denoted by the same symbol.

A main unsolved problem in representation theory of Lie groups is the classification of $\widehat{G}$, the \emph{unitary dual} of $G$, which consists of all the equivalence classes of irreducible unitary representations of $G$. Due to Harish-Chandra's foundational work, to classify $\widehat{G}$, it suffices to understand all irreducible $(\frg, K)$ modules having a positive definite invariant Hermitian form. To arrive at $\widehat{G}$, one has to deal with (great many) non-unitary representations. Indeed, let us denote by $\widehat{G}_{\rm admi}$ the set of all the irreducible admissible $(\frg, K)$ modules (up to equivalence), and by  $\widehat{G}_{\rm Herm}$ the members of $\widehat{G}_{\rm admi}$ having invariant Hermitian forms. Then one has that
$$
\widehat{G} \subset \widehat{G}_{\rm Herm} \subset \widehat{G}_{\rm admi}.
$$
Note that $\widehat{G}_{\rm admi}$ was classified by Langlands in 1970s, and $\widehat{G}_{\rm Herm}$ was classified by Knapp and Zuckerman later. To sieve out unitary representations from $\widehat{G}_{\rm Herm}$, one has to compute the signature of the Hermitian form. An elementary but effective tool on this aspect is Parthasarathy's Dirac operator inequality \cite{Pa2}.  Indeed, the \emph{Dirac operator} \cite{Pa} is defined as
\begin{equation}\label{Dirac-operator}
D:=\sum_{i=1}^{n} Z_i\otimes Z_i\in U(\frg)\otimes C(\frp)
\end{equation}
Here $\{Z_i\}_{i=1}^{n}$ is an orthonormal basis of $\frp_0$ with respect to the inner product on $\frp_0$ induced by $B(\cdot, \cdot)$, while $U(\frg)$ is the universal enveloping algebra and $C(\frp)$ is the Clifford algebra. Note that $D$ is independent of the choice of $\{Z_i\}_{i=1}^{n}$.

Using a formula of $D^2$, which is a natural Laplacian on $G$, one can deduce
the \emph{Dirac inequality} as follows
\begin{equation}\label{Dirac-inequality-original}
\|\gamma+\rho_c\| \geq \|\Lambda\|,
\end{equation}
where $\gamma$ is the highest weight of any $\widetilde{K}$-type in $\pi\otimes S_G$. Here $S_G$ is an irreducible module of $C(\frp)$ and $\widetilde{K}$ is the spin covering group of $K$. Namely, it consists of elements $(k, s)$ in $K\times {\rm Spin}(\frp_0)$ such that $Ad(k)=p(s)$, where $Ad: K\to {\rm SO}(\frp_0)$ is the adjoint map and $p: {\rm Spin}(\frp_0)\to {\rm SO}(\frp_0)$ is the double covering map.

Let $\pi$ be an irreducible admissible $(\frg, K)$ module. Then $D$ acts on $\pi\otimes S_G$.
To further sharpen the Dirac inequality, thus to have a better understanding of $\widehat{G}$, in 1997, Vogan introduced the \emph{Dirac cohomoloy} of  $\pi$ as the following $\widetilde{K}$ module
\begin{equation}\label{Dirac-cohomology}
H_D(\pi)=\ker D/\ker D\cap {\rm Im} D.
\end{equation}
A fundamental result in the study of Dirac cohomology is the following Vogan conjecture proven by Huang and Pand\v zi\'c in 2002.
\begin{thm}\label{thm-HP}\emph{(\cite{HP})}
Let $\pi$ be an irreducible $(\frg, K)$ module with infinitesimal character $\Lambda$. Suppose that $H_D(\pi)$ is non-zero. Then for any $\widetilde{K}$-type $E_{\gamma}$ in $H_D(\pi)$, there exists $w\in W(\frg, \frt_f)$ such that $w(\gamma+\rho_c)=\Lambda$.
\end{thm}

Now we assume that $\pi$ is unitary. Then it is typically infinite-dimensional. Since the Weyl group is finite, the above theorem says that to compute $H_D(\pi)$, it suffices to look at \emph{finitely many} $K$-types of $\pi$. These $K$-types will be analyzed and called \emph{spin-lowest $K$-types} in Section \ref{sec-structure}. We have that $\ker D\cap {\rm Im} D=0$, and  that
$$
H_D(\pi)=\ker D=\ker D^2.
$$
Let us collect the \emph{Dirac series} of $G$---all the irreducible unitary $(\frg, K)$ modules $\pi$ with non-vanishing Dirac cohomology---as $\widehat{G}^d$. Then of course we have
$$
\widehat{G}^d\subset \widehat{G}.
$$
On one hand, as we shall explain in Section \ref{sec-structure}, the Dirac series are exactly the members of $\widehat{G}$ such that Dirac inequality holds on some $K$-types of $\pi$. Thus we can expect Dirac series of $G$ to bear certain nice properties. A recent research announcement of Barbasch and Pand\v zi\'c \cite{BP19} indicates that one can use Dirac series to construct automorphic forms. On the other hand, by Theorem \ref{thm-HP}, Theorem A of \cite{D17}, Theorem 1.2 of \cite{D20} and its further sharpening, we can classify the Dirac series of $G$ using a relatively small workload. This will approximate the entire unitary dual of $G$.

\begin{thm}\label{thm-main} The set $\widehat{E_{8(-24)}}^d$ consists of $211$ FS-scattered representations whose spin-lowest $K$-types are u-small, and $3766$ strings of representations. Moreover, each spin-lowest
$K$-type of any Dirac series of $E_{8(-24)}$ occurs exactly once.
\end{thm}

The basic tools leading to the above theorem are \texttt{atlas} \cite{ALTV} and a finite way of presenting the infinite set $\widehat{G}^d$ \cite{D20}. Both of them, in particular, the notion of FS-scattered representation, will be introduced in Section \ref{sec-pre}. The notion unitarily small (\emph{u-small} for short) $K$-type was due to Salamanca-Riba and Vogan \cite{SV}, and that of spin-lowest $K$-type was due to the second author. Both of them will be recalled in Section \ref{sec-structure}.

The paper is organized as follows: Section 2 prepares some preliminaries, while Section \ref{sec-structure} reviews the structure of $E_{8(-24)}$. Section \ref{sec-HJ} further sharpens the Helgason-Johnson bound for
$E_{8(-24)}$, then Section \ref{sec-EIX-DS} illustrates how to pin down all its Dirac series.
Section \ref{sec-examples} gives a few examples. Finally, Section \ref{sec-appendix} is an appendix presenting all the FS-scattered representations.

\section{Preliminaries}\label{sec-pre}

We continue with the notations in the introduction.

\subsection{A brief introduction to the software \texttt{atlas}}
A recent important progress on understanding $\widehat{G}$ is an algorithm  testing the unitarity of any irreducible $(\frg, K)$ module. It is given by Adams, van Leeuwen, Trapa and Vogan \cite{ALTV}, and has software realization \texttt{atlas}. This section aims to briefly introduce \texttt{atlas}. More details will be recalled later when necessary.

A basic question in using \texttt{atlas} is how to get some irreducible admissible $(\frg, K)$ modules $\pi$. A quick way to do this is using infinitesimal character. For instance, the following command lists all the irreducible representations of $SU(2, 1)$ having infinitesimal character $\rho(G)=[2, 1]$.
\begin{verbatim}
G:SU(2,1)
set all=all_parameters_gamma(G, rho(G))
void: for p in all do prints(p, "  ",support(x(p)), " ", is_unitary(p)) od
final parameter(x=5,lambda=[2,1]/1,nu=[2,1]/1)  [0,1]  true
final parameter(x=4,lambda=[2,1]/1,nu=[1,-1]/2) [0]    true
final parameter(x=3,lambda=[2,1]/1,nu=[1,2]/2)  [1]    true
final parameter(x=2,lambda=[2,1]/1,nu=[0,0]/1)  []     true
final parameter(x=1,lambda=[2,1]/1,nu=[0,0]/1)  []     true
final parameter(x=0,lambda=[2,1]/1,nu=[0,0]/1)  []     true
\end{verbatim}
One sees that the Langlands parameter of $\pi$ has three parts: a KGB element \texttt{x}, vectors $\lambda$ and $\nu$.

Firstly, let us explain the KGB element. Let $G(\bbC)$ be a complex connected simple algebraic group with finite center. Let $\sigma$ be a real form of $G(\bbC)$. Namely, $\sigma$ is an antiholomorphic Lie group automorphism of $G(\bbC)$ such that $\sigma^2={\rm Id}$. Let $\theta$ be the involutive algebraic automorphism of $G(\bbC)$ corresponding to $\sigma$ via the Cartan theorem (see Theorem 3.2 of \cite{ALTV}). Let $H(\bbC)$ be a maximal connected abelian subgroup of $G(\bbC)$ consisting of diagonalizable matrices. Let
$$
X^*:={\rm Hom}_{alg}(H(\bbC), \bbC^\times)
$$
be the character lattice of $H(\bbC)$, which is the group of algebraic homomorphisms from $H(\bbC)$ to $\bbC$. Choose a Borel subgroup $B(\bbC)$ of $G(\bbC)$ containing $H(\bbC)$. Let $l=\dim_{\bbC}\, \frh$ and let $0, 1, 2, \dots, l-1$ be a labeling of the simple roots of $\Delta(\frb, \frh)$. Let $K(\bbC)$ be $G(\bbC)^\theta$. Now a KGB element \texttt{x} stands for a $K(\bbC)$-orbit of the Borel variety $G(\bbC)/B(\bbC)$. A KGB element \texttt{x} carries a lot of information. Here we only mention that the command \texttt{support(x)} gives its support, which is a subset of $\{0, 1, \dots, l-1\}$. Moreover, the command $\texttt{involution(x)}$ gives an $l\times l$ matrix $\theta_{\texttt{x}}$ which specializes the Cartan involution $\theta$. The vector $\lambda$ lives in $X^*+\rho$ and $\nu\in (X^*)^{-\theta}\otimes_{\bbZ}\bbC$.

Now the representation $\pi$ is presented by a final parameter $(\texttt{x}, \lambda, \nu)$. We say $\pi$ is \emph{fully supported} (\emph{FS} for short) if $\texttt{support(x)}=\{0, 1, \dots, l-1\}$. Note that the infinitesimal character of $\pi$ reads as
\begin{equation}\label{inf-char-atlas}
\Lambda=\frac{1}{2}(1+\theta_{\texttt{x}})\lambda +\nu\in\frh^*.
\end{equation}
One sees that there are six irreducible representations of $SU(2, 1)$ with infinitesimal character $\rho(G)$, among which one is fully supported.

The online seminars \cite{At} give a  very good introduction to the ongoing project \texttt{atlas}.

\subsection{Presenting the infinite set $\widehat{G}^d$}
As in the $SU(2, 1)$ example in the previous section, \texttt{atlas} organizes the representations with a given infinitesimal character according to their supports. A result of Voan \cite{Vog84}, rephrased by Paul in the \texttt{atlas} language \cite{Paul}, says that whenever $\pi$ is not fully supported,  it can be cohomologically induced from a representation $\pi_{L(\texttt{x})}$ from the weakly good range in a canonical way. Here $\pi_{L(\texttt{x})}$ is the Levi subgroup of the $\theta$-stable parabolic subalgebra
$$
\frq(\texttt{x})=\frl(\texttt{x})+\fru(\texttt{x})
$$
defined by the pair $(\texttt{support(x)}, \texttt{x})$, and the representation $\pi_{L(\texttt{x})}$ has final parameter $(\texttt{y}, \lambda-\rho(\fru(\texttt{x})), \nu)$, where \texttt{y} is the unique KGB element of $L(\texttt{x})$ corresponding to the KGB element \texttt{x} of $G$ and $\rho(\fru(\texttt{x}))$ is the half-sum of roots in $\Delta(\fru(\texttt{x}), \frh_f)$.

Let us continue with the previous example, and illustrate how to find the inducing representation $\pi_{L(\texttt{x})}$ from $\pi$:
\begin{verbatim}
set p=all[1]
p
Value: final parameter(x=4,lambda=[2,1]/1,nu=[1,-1]/2)
set (Q,q)=reduce_good_range(p)
set L=Levi(Q)
Q
Value: ([0],KGB element #4)
L
Value: connected quasisplit real group with Lie algebra 'sl(2,R).u(1)'
q
Value: final parameter(x=2,lambda=[1,-1]/2,nu=[1,-1]/2)
rho_u(Q)
Value: [ 3, 3 ]/2
goodness(q,G)
Value: "Good"
q=trivial(L)
Value: true
\end{verbatim}
Therefore, as guaranteed by the main result of Salamanca-Riba \cite{Sa}, the representation \texttt{p} is an $A_{\frq}(\lambda)$ module.

It turns out that one can also organize the Dirac series $\widehat{G}^d$ according to their supports. Indeed, we call $\pi\in \widehat{G}^d$ \emph{FS-scattered} if it is fully supported. As shown in \cite{D20}, there are \emph{finitely many} FS-scattered representations of $G$, and we can organize those non-fully-supported Dirac series of $G$ into finitely many \emph{strings}, with each representation in a string having the same KGB element and the same $\nu$. The representations in a string come from \emph{one} FS-scattered representation of $L(\texttt{x})$ (tensored with varying unitary characters) via the canonical cohomological induction described above. In particular, for any fixed $G$, we can present its Dirac series in a finite way.

It is interesting to note that even for classical groups such as \cite{BDW} and\cite{DW}, the above method is helpful: one determines the small rank cases, and then tries to find the general pattern and give a rigorous proof.

\section{The structure of $E_{8(-24)}$}\label{sec-structure}

Although many results quoted in this section work in a much wider setting, to focus on the main target, we fix $G$ as the connected simple linear real exceptional Lie group \texttt{E8\_q} in \texttt{atlas} henceforth. It has trivial center. It is \emph{not} simply connected. Indeed,
\begin{verbatim}
set K=K_0(G)
K
Value: compact connected real group with Lie algebra 'e7.su(2)'
print_Z(K)
Group is semisimple
center=Z/2Z
\end{verbatim}
The output says that $K$ has center $\bbZ/2\bbZ$. However, its universal covering group $\widetilde{K}$ has center $\bbZ/2\bbZ \times \bbZ/2\bbZ$. Therefore, $\widetilde{K}$ is a double cover of $K$. Thus by the Cartan decomposition $G=K\exp(\frp_0)$, the universal covering group $\widetilde{G}$ is also a double cover of $G$.

The group $G$ is equal rank, i.e., $\frt_f=\frh_f$. The Lie algebra $\frg_0$ is denoted by \texttt{EIX} in Appendix C of \cite{Kn}. We fix a Vogan diagram of $\frg_0$ as in Fig.~\ref{Fig-EIX-Vogan}. In particular, we have fixed a positive root system $(\Delta^+)^{(0)}(\frg, \frh_f):=\Delta^+(\frg, \frh_f)$ for $\Delta(\frg, \frh_f)$. Here the simple roots are $\alpha_1=\frac{1}{2}(1, -1,-1,-1,-1,-1,-1,1)$, $\alpha_2=e_1+e_2$ and $\alpha_i=e_{i-1}-e_{i-2}$ for $3\leq i\leq 8$. Let $\zeta_1, \dots, \zeta_8$ be the corresponding fundamental weights. We will use them as a basis to express the \texttt{atlas} parameters $\lambda$, $\nu$ and the infinitesimal character. More precisely, in such cases, $[a, b, c, d, e, f, g, h]$ stands for the vector $a\zeta_1+\cdots+ h \zeta_8$.

Note that
$$
(\Delta^+)^{(0)}(\frg, \frh_f)=\Delta^+(\frk, \frh_f)\cup (\Delta^+)^{(0)}(\frp, \frh_f).
$$
We emphasize that $\alpha_8$, the unique painted root, is non-compact. In other words, $\alpha_8\in (\Delta^+)^{(0)}(\frp, \frt_f)$. All the other simple roots are compact. Let $\beta=2\alpha_1+3\alpha_2+ 4\alpha_3+ 6\alpha_4+ 5\alpha_5+ 4\alpha_6+ 3\alpha_7 +\alpha_8$. Then it is the highest root of $(\Delta^+)^{(0)}(\frp, \frt_f)$, and $\frp\cong V_{\beta}$ as $\frk$ modules. Note that
$$
-\dim \frk +\dim \frp=-136+112=-24.
$$
Thus the group $G$ is also called $E_{8(-24)}$in the literature.

\begin{figure}[H]
\centering
\scalebox{0.65}{\includegraphics{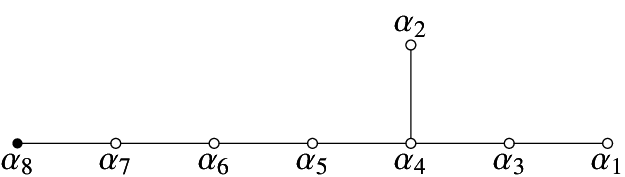}}
\caption{The Vogan diagram for $E_{8(-24)}$}
\label{Fig-EIX-Vogan}
\end{figure}

We fix a Dynkin diagram for $\Delta^+(\frk, \frt_f)$ in Fig.~\ref{Fig-EIX-Dynkin}, where the simple roots $\gamma_i=\alpha_i$ for $1\leq i\leq 7$ and $\gamma_8=\beta +\alpha_8$. Let $\varpi_1, \dots, \varpi_8$ be the corresponding fundamental weights. Let $\rho_c$ be the half-sum of roots in $\Delta^+(\frk, \frt_f)$. An irreducible representation of $\frk$, that is a $\frk$-type, is parameterized by its highest weight
$$
\mu=a\varpi_1+b\varpi_2+c\varpi_3+d\varpi_4+e\varpi_5+f\varpi_6+g\varpi_7+h\varpi_8,
$$
where $a,b,\dots, h$ are non-negative integers. We will write
\begin{equation}\label{k-type-hwt}
\mu=[a, b, c, d, e, f, g, h]
\end{equation}
for short. That is, we use $\varpi_1, \dots, \varpi_8$ as a basis to express the highest weight of a $\frk$-type.  For instance, we have that
$$
\rho_c=[1,1,1,1,1,1,1,1], \quad \beta=[0, 0, 0, 0, 0, 0, 1, 1].
$$
Similar notation applies to $K$-types and $\widetilde{K}$-types.

\begin{figure}[H]
\centering
\scalebox{0.6}{\includegraphics{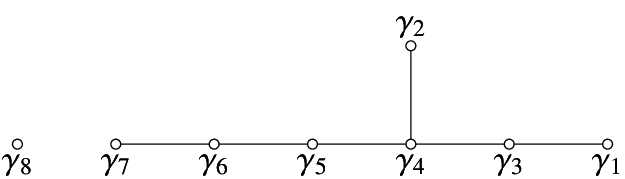}}
\caption{The Dynkin diagram for $\Delta^+(\frk, \frt_f)$}
\label{Fig-EIX-Dynkin}
\end{figure}

Note that a $\frk$-type $\mu=[a, b, \dots, h]$ becomes a $K$-type if and only if
\begin{equation}\label{K-type-condition}
b+e+g+h \mbox{ is even. }
\end{equation}
Since the longest element $w_0^K$ of $W(\frk, \frt_f)$ equals $-1$, the contragredient $\frk$-type of $E_{\mu}$ is $E_{\mu}$ itself.

One can shift the coordinates of a $K$-type from the \texttt{atlas} fashion to that of \eqref{k-type-hwt} as follows: use \texttt{KGB(G, 119)} to obtain the highest weight $(y_1, y_2, \dots, y_8)$ of a $K$-type. Then it becomes
\begin{equation}\label{k-type-hwt-from-atlas}
[y_1, y_2, y_3, y_4, y_5, y_6, y_7,2 y_1+3 y_2+4 y_3+6 y_4+5 y_5+4 y_6+3 y_7+2 y_8].
\end{equation}
For instance,
\begin{verbatim}
G:E8_q
set p=parameter(G,67078,[2,3,3,-2,1,1,3,0],[8,5,5,-5,0,0,5,5]/2)
print_branch_irr_long(p,KGB(G,119),300)
m  x    lambda                 hw                             dim  height
1 669  [0,1,1,-1,1,-1,1,0] KGB element #119 [0,0,0,0,0,0,0,4] 9    211
1 398  [0,1,1,-1,1,0,0,0]  KGB element #119 [0,0,0,0,0,0,1,3] 560  266
\end{verbatim}
The above $K$-types are $[0,0,0,0,0,0,0,8]$ and $[0,0,0,0,0,0,1,9]$, respectively.

\subsection{Positive root systems and the set $W(\frg, \frt_f)^1$}
Let $\caC^{(0)}_\frg$ (resp., $\caC$) be the closed dominant Weyl chamber for $(\Delta^+)^{(0)}(\frg, \frh_f)$(resp., $\Delta^+(\frk, \frh_f)$). Denote
\begin{equation}\label{W1}
W(\frg, \frt_f)^1=\{w\in W(\frg, \frt_f)\mid w \caC^{(0)}_\frg\subseteq \caC\}.
\end{equation}
By a result of Kostant \cite{Ko}, the multiplication map gives a bijection from $W(\frg, \frt_f)^1 \times W(\frk, \frt_f)$ onto $W(\frg, \frt_f)$. In particular, if follows that
$$
\#W(\frg, \frt_f)^1=\frac{\#W(\frg, \frt_f)}{\# W(\frk, \frt_f)}=\frac{2^{14}\cdot 3^5 \cdot 5^2 \cdot 7}{2^{10}\cdot 3^4 \cdot 5 \cdot 7 \times 2}=120.
$$
Let us enumerate the elements of $W(\frg, \frt_f)^1$ as $w^{(0)}=e$, $w^{(1)}, \dots, w^{(119)}$. Then all the positive root systems of $\Delta(\frg, \frt_f)$ containing $\Delta^+(\frk, \frt_f)$ are exactly
$$
(\Delta^+)^{(j)}(\frg, \frt_f):=\Delta^+(\frk, \frt_f) \cup w^{(j)} (\Delta^+)^{(0)}(\frp, \frt_f), \quad 0\leq j\leq 119.
$$
Let us denote by $\rho_n^{(j)}$ the half sum of positive roots in $w^{(j)} (\Delta^+)^{(0)}(\frp, \frt_f)$ and put $\rho^{(j)}=\rho_c+\rho_n^{(j)}$.

\subsection{Lambda norm and infinitesimal character}
Given a $K$-type $E_{\mu}$, choose an index $j$ such that $\mu+2\rho_c\in\caC^{(j)}_\frg$. Put
\begin{equation}\label{def-lambdamu}
\lambda_a(\mu)=P(\mu+2\rho_c-\rho^{(j)}),
\end{equation}
which stands for the unique point in $\caC^{(j)}_\frg$ that is closest to $\mu+2\rho_c-\rho^{(j)}$. It turns out $\lambda_a(\mu)$ is well-defined. That is, it is independent of the choice of an allowable index $j$. Moreover, the \emph{lambda norm} due to Vogan \cite{V81} can be equivalently described \cite{Ca} as
\begin{equation}\label{def-lambdanorm}
\|\mu\|_{\rm lambda}=\|\lambda_a(\mu)\|.
\end{equation}
Now the lambda norm of an irreducible $(\frg, K)$ module $\pi$ is defined as
$$
\|\pi\|_{\rm lambda}:=\min_{\mu} \|\mu\|_{\rm lambda},
$$
where $\mu$ runs over all the $K$-types of $\pi$. We call $\mu$ a \emph{lowest $K$-type} of $\pi$ (LKT for short) if it occurs in $\pi$ and $\|\mu\|_{\rm lambda}=\|\pi\|_{\rm lambda}$.

Let $\mu$ be one of the LKTs of $\pi$, then a representative of the infinitesimal character of $\pi$ can be chosen as
\begin{equation}\label{inf-char}
\Lambda=(\lambda_a(\mu), \nu)\in \frh^*=\frt^*+\fra^*.
\end{equation}
Here $H=TA$ is a maximally split $\theta$-stable Cartan subgroup of the isotropy subgroup $G(\lambda_a(\mu))$. Note that $\frt\subset \frt_f$.

\subsection{Spin norm and Dirac inequality}
The \emph{spin norm} of a $K$-type $E_{\mu}$ \cite{D13} now specializes as
\begin{equation}\label{def-spin-norm}
\|\mu\|_{\rm spin}:=\min_{0\leq j\leq 119} \|\{\mu-\rho_n^{(j)}\}+\rho_c\|
\end{equation}
Here $\{\mu-\rho_n^{(j)}\}$ is the unique element in $\caC$ to which $\mu-\rho_n^{(j)}$ is conjugate under the action of $W(\frk, \frt_f)$. Now the spin norm of an irreducible $(\frg, K)$ module $\pi$ is defined as
$$
\|\pi\|_{\rm spin}:=\min_{\mu} \|\mu\|_{\rm spin},
$$
where $\mu$ runs over all the $K$-types of $\pi$. We call $\mu$ a \emph{spin-lowest $K$-type} of $\pi$ (spin LKT for short) if it occurs in $\pi$ and $\|\mu\|_{\rm spin}=\|\pi\|_{\rm spin}$.

Given a $K$-type $E_{\mu}$, using knowledge of the PRV component \cite{PRV}, one has that
$$
\|\gamma+\rho_c\| \geq \|\mu\|_{\rm spin},
$$
where $\gamma$ is the highest weight of any $\widetilde{K}$-type of $E_{\gamma}\otimes S_G$.
Assume that the $(\frg, K)$ module $\pi$ is \emph{unitary}.
The Dirac inequality \eqref{Dirac-inequality-original} can now be rephrased as follows:
\begin{equation}\label{Dirac-inequality}
\|\mu\|_{\rm spin} \geq \|\Lambda\|,
\end{equation}
where $\mu$ is the highest weight of any $K$-type of $\pi$. Therefore,
$$
\|\pi\|_{\rm spin}\geq \|\Lambda\|
.
$$
Moreover, it follows from Theorem 3.5.12 of \cite{HP2} that
$H_D(\pi)$ is non-zero if and only if $\|\pi\|_{\rm spin}=\|\Lambda\|$, and in such a case, it is exactly the spin LKTs that contribute to $H_D(\pi)$.

\subsection{Vogan pencil and the u-small convex hull}
By a result of Vogan \cite{Vog80}, since $G/K$ is not Hermitian symmetric, if $\mu$ is a $K$-type occurring in any \emph{infinite-dimensional} irreducible $(\frg, K)$ module $\pi$, then the $K$-type $\mu+n\beta$ shall occur in $\pi$ for any non-negative integer $n$. We call
\begin{equation}
P(\mu):=\{\mu+n\beta\mid n\in\bbZ_{\geq 0}\}
\end{equation}
the \emph{Vogan pencil} starting from $\mu$. We also put
\begin{equation}
{\rm MP}(\mu):=\min_{n\in\bbZ_{\geq 0}} \|\mu+n\beta\|_{\rm spin}
\end{equation}
Namely, ${\rm MP}(\mu)$ is the minimum spin norm of all the $K$-types on $P(\mu)$.

As in \cite{SV},
we call a $K$-type $E_{\mu}$ \emph{u-small} if $\mu$ lies in the convex hull generated by the points $2w\rho_n^{(j)}$, where $w\in W(\frk, \frt_f)$ and $0\leq j\leq s-1$. Otherwise, we will say that $\mu$ is \emph{u-large}.

By Section 6.7 of \cite{D20}, when $\mu$ is u-large, one has ${\rm MP}(\mu)=\|\mu\|_{\rm spin}$; when $\mu$ is u-small, to obtain ${\rm MP}(\mu)$, it suffices to compute the spin norm of those (finitely many) $\mu+n\beta$ which are u-small.

\begin{example}\label{exam-pencil}
Consider $P(0)$, the Vogan pencil starting from the trivial $K$-type. We have that $15\beta$ is u-small, while $16\beta$ is u-large. The spin norms of $n\beta$ for $0\leq n\leq 17$ are as follows:
\begin{align*}
&\sqrt{620}, \sqrt{564}, \sqrt{512}, \sqrt{464}, \sqrt{420}, \sqrt{380}, \sqrt{348}, \sqrt{320}, \sqrt{296}, \sqrt{276}, \\
&\sqrt{260}, \sqrt{300}, \sqrt{344}, \sqrt{392}, \sqrt{444}, \sqrt{500}, \sqrt{560}, \sqrt{624}.
\end{align*}
Therefore, ${\rm MP}(0)=\sqrt{260}$ and it is attained at $10\beta$.
\hfill\qed
\end{example}
\begin{rmk}\label{rmk-exam-pencil}
When $\mu$ is u-small, we guess that the spin norm  should decrease firstly, and then increase along $P(\mu)$.
\end{rmk}

\section{Sharpening the the Helgason-Johnson bound for $E_{8(-24)}$}\label{sec-HJ}
\begin{prop}\label{prop-previous-HJ}
Let $G$ be $E_{8(-24)}$. Let $\pi$ be an irreducible unitary $(\frg, K)$ module  whose infinitesimal character $\Lambda$ is given as \eqref{inf-char}. Then
\begin{itemize}
\item [(a)] $\|\nu\|\leq \sqrt{620}=\|\rho(G)\|$;
\item [(b)] $\|\nu\|\leq \sqrt{\frac{723}{2}}$ if $\pi$ is infinite-dimensional.
\end{itemize} 	
\end{prop}

Item (a) above originates from Helgason and Johnson \cite{HJ} in 1969. Item (b) is obtained in \cite{D20} recently.  To save time in doing non-unitarity test, let us enhance it further by adopting suitable assumptions. We say the infinitesimal character $\Lambda=[a, b, \dots, h]=a\xi_1+b\xi_2+\cdots+h\xi_8$ is \textbf{HP integral} if $a, b, \dots, h\in\bbZ_{\geq 0}$, and that
\begin{itemize}
\item[$\bullet$] $a+c>0$, $b+d>0$, $c+d>0$, $d+e>0$, $e+f>0$, $f+g>0$, $g+h>0$;
\item[$\bullet$] $d+f+h>0$, $a+b+e+g>0$, $a+b+e+h>0$, $a+b+f+h>0$;
\item[$\bullet$] $b+c+e+g>0$, $b+c+e+h>0$, $b+c+f+h>0$.
\end{itemize}
The above terminology is justified as follows: If $\Lambda$ is not HP integral, then by Theorem \ref{thm-HP}, any irreducible $(\frg, K)$ module $\pi$ with infinitesimal character $\Lambda$ (if exists) will not be a Dirac series. The specific conditions above are deduced by knowledge of $W(\frg, \frt_f)^1$. Indeed, we firstly express $w^{(j)}\Lambda$ in terms of $\varpi_1, \varpi_2, \cdots, \varpi_l$ for all $w^{(j)}\in W(\frg, \frt_f)^1$. These $120$ vectors turn out to be
\begin{align*}
&[a,b,c,d,e,f,g, 2a+3b+4c +6d+5e+4f+3g+2h],\\
&[a,b,c,d,e,f,g+h, 2a+3b+4c+6d+5e+4f+3g+h],\\
&\dots,\\
&[f,b,e,d,c,a,b+c+2d+2e+2f+2g+h,h].
\end{align*}
Then we check that they \emph{all} have zero coordinate(s) whenever any one of the above conditions, say $a+c>0$, fails. To sum up, to search for the Dirac series, it suffices to consider HP integral infinitesimal characters.
	
Before moving on, let us single out a few irreducible unitary representations.
\begin{example}\label{exam-HJ}
The following three irreducible representations are all unitary:
\begin{verbatim}
G:E8_q
set p1=parameter(G,65831,[1,4,1,-3,4,0,1,4],[0,9,0,-9,9,-1,2,9]/2),
set p2=parameter(G,67055,[1,2,1,-1,3,-1,4,1],[0,1,0,-1,4,-3,4,1]),
set p3=parameter(G,67078,[3,2,2,-1,1,1,2,1],[8,5,5,-5,0,0,5,5]/2)
\end{verbatim}
All of them have non-zero Dirac cohomology. The statistics $\|\nu\|^2$ for them are $\frac{443}{2}$, $246$ and $\frac{723}{2}$, respectively. Moreover, they have GK dimensions $46$, $46$ and $29$, respectively. Here GK dimension means the Gelfand-Kirillov dimension in \cite{Vog78}. It can be computed in \texttt{atlas} in the following way:
\begin{verbatim}
GK_dim(p3)
Value: 29
\end{verbatim}
Note that $(\rho, \gamma_8^\vee)=29$. Here $\gamma_8$ is as in Figure \ref{Fig-EIX-Dynkin}: it is the highest root of $(\Delta^+)^{(0)}(\frg, \frt_f)$. Indeed, the above \texttt{p3} is the unique minimal representation of $E_{8(-24)}$.

The three representations \texttt{p1},  \texttt{p2} and \texttt{p3} are recorded in Tables \ref{table-EIX-11101011-part1}, \ref{table-EIX-11101011-part1}, and \ref{table-EIX-11101111}, respectively. Note that the \texttt{lambda} parameter of a representation may have different forms. For instance,
\begin{verbatim}
set q1=parameter(KGB(G,65831),[1,4,1,-3,4,0,1,4], [0,9,0,-9,9,-1,2,9]/2)
p1=q1
Value: true
\end{verbatim} \hfill\qed
\end{example}

Now let us further sharpen the Helgason-Johnson bound.
\begin{prop}\label{prop-HJ}
Let $G$ be $E_{8(-24)}$. Except for the three representations described in Example \ref{exam-HJ}, there is no irreducible unitary $(\frg, K)$ module $\pi$ with a HP integral infinitesimal character $\Lambda$ such that
\begin{equation}\label{further-HJ-bound}
\frac{425}{2}	\leq \|\nu\|^2 \leq \frac{723}{2}.
\end{equation}		
\end{prop}
	
\begin{proof}
Take a LKT $\mu$ of $\pi$. Then the infinitesimal character of $\pi$ has the form $\Lambda=(\lambda_a(\mu), \nu)$ as in \eqref{inf-char}. By the Dirac inequality \eqref{Dirac-inequality}, we have
		$$
		\|\Lambda\|^2=\|\lambda_a(\mu)\|^2 +\|\nu\|^2\leq \|\mu\|^2_{\rm spin}.
		$$
		Therefore,
		\begin{equation}\label{nu-bound}
			\|\nu\|^2\leq \|\mu\|^2_{\rm spin}- \|\mu\|^2_{\rm lambda}.
		\end{equation}
		As computed in Section 5.10 of \cite{D20}, we have that $\max\,\{A_j\mid 0\leq j\leq 119\}=212$. We refer the reader to Section 3 of \cite{D20} for the precise meaning of these $A_j$. It follows that for any u-large $K$-type $\mu$, one has that
		$$
		\|\mu\|_{\rm spin}^2 - \|\mu\|_{\rm lambda}^2\leq 212.
		$$
Since $\|\nu\|^2\geq 212.5$ by our assumption \eqref{further-HJ-bound}, we conclude from \eqref{nu-bound} that $\mu$ can \emph{not} be u-large.
		
There are $294660$ u-small $K$-types in total. Among them, only $249$ have the property that
$$
212.5\leq \|\mu\|_{\rm spin}^2 - \|\mu\|_{\rm lambda}^2.
$$
Let us collect these $249$ u-small $K$-types as \texttt{Certs}. Some of its members are
$$
[0, 0, 0, 0, 0, 0, 0, 0], [0, 0, 0, 0, 0, 1, 0, 0], [1, 0, 0, 0, 0, 0, 0, 0], \dots, [0, 0, 0, 0, 0, 0, 6, 18].
$$
We compute that
		$14\leq \|\lambda_a(\mu)\|^2\leq 178$ for any $\mu\in \texttt{Certs}$. Therefore,
		\begin{equation}\label{can-Lambda}
			14+ 212.5 \leq \|\Lambda\|^2=\|\lambda_a(\mu)\|^2+\|\nu\|^2\leq 178 + 361.5.
		\end{equation}		
The right hand side above uses item (b) of Proposition \ref{prop-previous-HJ}. There are $919$ HP integral $\Lambda$s meeting the requirement \eqref{can-Lambda}. We collect them as $\Omega$.

		Now a direct search using \texttt{atlas} says that there are $4256$ irreducible  representations $\pi$ such that $\Lambda\in \Omega$, that $\pi$ has a LKT which is a member of \texttt{Certs}, and that
$$
\min_{\mu} {\rm MP}(\mu)\geq \|\Lambda\|,
$$
where $\mu$ runs over all the LKTs of $\pi$. It turns out that only three of them are unitary. They are described in Example \ref{exam-HJ}. This finishes the proof.	
	\end{proof}

\section{Dirac series of $E_{8(-24)}$}\label{sec-EIX-DS}

This section aims to classify the Dirac series of $E_{8(-24)}$. We will organize them into two parts:  finitely many FS-scattered representations, and finitely many strings.
Let $\Lambda=[a, b, c, d, e, f, g, h]$ be a \emph{dominant} infinitesimal character. Denote by $\Pi(\Lambda)$ (resp., $\Pi_{\rm FS}(\Lambda)$) the set of all the (resp. fully supported) irreducible representations with infinitesimal character $\Lambda$. We pick those unitary ones in $\Pi_{\rm FS}(\Lambda)$ as $\Pi_{\rm FS}^{\rm u}(\Lambda)$.

\subsection{FS-scattered representations of $E_{8(-24)}$}\label{sec-FS-EIX}

By the discussions in Section \ref{sec-HJ}, to sieve out all the FS-scattered Dirac series of $E_{8(-24)}$ other than the trivial one and those described in Example \ref{exam-HJ}, it suffices to go through the  irreducible \emph{unitary} representations of $E_{8(-24)}$ with infinitesimal character $\Lambda=[a, b, c, d, e, f, g, h]$ such that
\begin{itemize}
\item[$\bullet$] $\min\{a, b, c, d, e, f, g, h\}=0$;
\item[$\bullet$]  $\Lambda$ is HP integral (see Section \ref{sec-HJ});
\item[$\bullet$] there exists a \emph{fully supported} KGB element \texttt{x} such that $\|\frac{\Lambda-\theta_{\texttt{x}} (\Lambda)}{2}\|^2 <\frac{425}{2}$.
\end{itemize}
Here the first requirement is justified by the main result of Salamanca-Riba \cite{SV}: otherwise $\Lambda$ would be strongly regular and any representation in $\Pi_{\rm FS}^{\rm u}(\Lambda)$ must be an $A_{\frq}(\lambda)$ module which is in the good range, while the third requirement comes from Proposition \ref{prop-HJ}.

We collect those $\Lambda$ as $\Phi$, and denote by $\Phi_N$ the members of $\Phi$ having maximal coordinate $N$. For instance, $\#\Phi_1=51$, $\#\Phi_2=3005$ and $\#\Phi_3=34933$. We present $\Phi_1$ as follows:
\begin{align*}
&[0, 1, 1, 0, 1, 0, 1, 0]; \\
& [0, 0, 1, 1, 0, 1, 1, 1], [0, 0, 1, 1, 1, 0, 1, 1], [0, 0, 1, 1, 1, 1,
0, 1], [0, 0, 1, 1, 1, 1, 1, 0], [0, 1, 1, 0, 1, 0, 1, 1], \\
&[0, 1, 1, 0, 1, 1, 0, 1], [0, 1, 1, 0, 1, 1, 1, 0], [0, 1, 1, 1, 0, 1, 0, 1], [0, 1, 1, 1, 0, 1, 1, 0],[0, 1, 1, 1, 1, 0, 1, 0], \\
&[1, 0, 0, 1, 0, 1, 1, 1], [1, 0, 0, 1, 1, 0, 1, 1], [1, 0, 0, 1, 1, 1, 0, 1], [1, 0, 0, 1, 1, 1, 1, 0],
[1, 0, 1, 1, 0, 1, 0, 1],\\
&[1, 0, 1, 1, 0, 1, 1, 0], [1, 0, 1, 1, 1, 0, 1, 0], [1, 1, 0, 1, 0, 1, 0, 1], [1, 1, 0, 1, 0, 1, 1, 0],
[1, 1, 0, 1, 1, 0, 1, 0],\\
&[1, 1, 1, 0, 1, 0, 1, 0];\\
&[0, 0, 1, 1, 1, 1, 1, 1], [0, 1, 1, 0, 1, 1, 1, 1], [0, 1, 1, 1, 0, 1,
1, 1], [0, 1, 1, 1, 1, 0, 1, 1], [0, 1, 1, 1, 1, 1, 0, 1], \\
&[0, 1, 1, 1, 1, 1, 1, 0], [1, 0, 0, 1, 1, 1, 1, 1], [1, 0, 1, 1, 0, 1, 1, 1], [1, 0, 1, 1, 1, 0, 1, 1], [1, 0, 1, 1, 1, 1, 0, 1], \\
&[1, 0, 1, 1, 1, 1, 1, 0], [1, 1, 0, 1, 0, 1, 1, 1], [1, 1, 0, 1, 1, 0, 1, 1], [1, 1, 0, 1, 1, 1, 0, 1], [1, 1, 0, 1, 1, 1, 1, 0], \\
&[1, 1, 1, 0, 1, 0, 1, 1], [1, 1, 1, 0, 1, 1, 0, 1], [1, 1, 1, 0, 1, 1, 1, 0], [1, 1, 1, 1,
0, 1, 0, 1], [1, 1, 1, 1, 0, 1, 1, 0], \\
&[1, 1, 1, 1, 1, 0, 1, 0];\\
&[0, 1, 1, 1, 1, 1, 1, 1], [1, 0, 1, 1, 1, 1, 1, 1], [1, 1, 0, 1, 1, 1, 1, 1], [1, 1, 1, 0, 1, 1, 1, 1], [1, 1, 1, 1, 0, 1, 1, 1], \\
&[1, 1, 1, 1, 1, 0, 1, 1], [1, 1, 1, 1, 1, 1, 0, 1], [1, 1, 1, 1, 1, 1, 1, 0].
\end{align*}

The following lemma has its morale from Vogan's fundamental parallelepiped conjecture (FPP). See also Conjecture 1.4 of \cite{DDH}, which can be viewed as a weak version FPP. It is worthy noting that the FPP has been confirmed by Wong recently \cite{Wong}.

\begin{lemma}\label{lemma-FPP}
 The set $\Pi_{\rm FS}^{\rm u}(\Lambda)$ of $E_{8(-24)}$ is empty for any $\Lambda\in\Phi_N$ whenever $N\geq 2$.
\end{lemma}

The lemma says that the infinitesimal character of any FS-scattered Dirac series of $E_{8(-24)}$ must lie in $\Phi_1$. All the FS-scattered Dirac series are presented according to their infinitesimal characters in Tables \ref{table-EIX-01101011}--\ref{table-EIX-11111111} in Section \ref{sec-appendix}.

Verifying the lemma is not an easy job for us. Actually, it involved a great amount of calculation thus five authors (and the computers that they can access) are included.

We check Lemma \ref{lemma-FPP} as follows: for any $\Lambda\in \Phi_2 \cup \Phi_3$, print a LKT $\mu$ of each member of $\Pi_{\rm FS}(\Lambda)$, then compute $\min_{n\geq 0} \|\mu+n\beta\|_{\rm spin}$ and compare it with $\|\Lambda\|$. Whenever the former is strictly smaller than the latter, we conclude that $\pi$ is not unitary. This method works well for most cases. Whenever it loses effect, we check the unitarity of $\pi$  using \texttt{atlas} directly.

We do not enumerate $\Phi_N$ for $N\geq 4$: it should be very challenging.  Instead, the following two observations can help us to reduce the workload:
\begin{itemize}
\item[(a)] although there are $51486$ fully supported KGB elements of $E_{8(-24)}$ in total,  many of them will not show up as the KGB element of any member of $\Pi_{\rm FS}(\Lambda)$ when $\Lambda$ is sepecific. For instance,
    $\#\Pi_{\rm FS}([0,1,1,0,1,0,1,0])=1860$ and only $1331$ KGB elements appear; $\#\Pi_{\rm FS}([1,1,1,1,1,1,1,0])= 23001$, and only $15428$ KGB elements appear.
\item[(b)] whenever $\Lambda$ and $\Lambda^\prime$ have the same non-zero coordinates, by translation functor (see Theorem 7.232 of \cite{KV}), there is a bijection from $\Pi_{\rm FS}(\Lambda)$ to $\Pi_{\rm FS}(\Lambda^\prime)$ preserving the KGB elements.
\end{itemize}
These observations and the structure of $\Phi_1$ hint us to define $51$ disjoint sets $\Phi(bceg)$, $\Phi(cdfgh)$, $\dots, \Phi(abcdefg)$. For instance, $\Phi(bceg)$ collects those $\Lambda$ in $\bigcup_{N\geq 4}\Phi_N$ such that $b, c, e, g$ are non-zero, that all other coordinates are zero, and that
$\|\frac{\Lambda-\theta_{\texttt{x}} (\Lambda)}{2}\|^2 <\frac{425}{2}$, where \texttt{x}
  is  a KGB element for certain representation in $\Pi_{\rm FS}([0,1,1,0,1,0,1,0])$.

It turns out that  $\#\Phi(bceg)=3459$, $\#\Phi(cdfgh)=19885$, $\#\Phi(abcdefg)=376150$ and so on. The (disjoint) union of these $51$ sets should be much smaller than $\cup_{N\geq 4}\Phi_N$.  To verify Lemma \ref{lemma-FPP}, now it remains to check that any representation in $\Pi_{\rm FS}(\Lambda)$ is non-unitary, where $\Lambda$ runs over $\Phi(bceg)\cup \Phi(cdfgh)\cdots\cup\Phi(abcdefg)$. Then we move on as in the case of $\Phi_2\cup \Phi_3$.

After the verification of Lemma \ref{lemma-FPP}, we test unitarity for the representations in $\Pi_{\rm FS}(\Lambda)$ with $\Lambda$ running over $\Phi_1$ (Section \ref{sec-HJ} is still helpful here). Then we look at the $K$-types of the unitary representations up to the height required by Theorem \ref{thm-HP}. The final results are presented in the Appendix. Some concrete examples will be given in Section \ref{sec-examples}.

\subsection{Counting the strings in $\widehat{E_{8(-24)}}^d$}\label{sec-EIX-esc}

We use the method of \cite{D20} to count the number of strings in $\widehat{E_{8(-24)}}^d$. We compute that
\begin{align*}
&N([0,1,2,3,4,5,6])=177, \quad N([0,1,2,3,4,5,7])=31, \quad N([0,1,2,3,4,6,7])=21, \\
&N([0,1,2,3,5,6,7])=12, \quad N([0,1,2,4,5,6,7])=1, \quad N([0,1,3,4,5,6,7])=17,\\
&N([0,2,3,4,5,6,7])=115, \quad N([1,2,3,4,5,6,7])=162.
\end{align*}
In particular, it follows that $N_7=536$. We also compute that
$$
N_0=120, \quad N_1=224, \quad N_2=322, \quad N_3=469, \quad N_4=628, \quad N_5=736, \quad N_6=731.
$$
Therefore, the total number of strings for $E_{8(-24)}$ is equal to
$$
\sum_{i=0}^{7} N_i=3766.
$$

We have built some files to facilitate the classification of the Dirac series of $E_{8(-24)}$. These files are about how to further sharpen the Helgason-Johnson bound, how to determine the spin LKTs of the FS-scattered representations, and how to count the number of strings.
They are available via the following link:
\begin{verbatim}
https://www.researchgate.net/publication/374506412_EIX-files
\end{verbatim}

\section{Examples}\label{sec-examples}

This section aims to look at some FS-scattered Dirac series carefully. In particular, we will compare Dirac cohomology and Dirac index. Note that \emph{Dirac index} is the formal difference of the even part and the odd part of Dirac cohomology. Namely,
\begin{equation}\label{Dirac-index}
{\rm DI}(\pi):=H_D^+(\pi)-H_D^-(\pi),
\end{equation}
which is a virtual $\widetilde{K}$-module. Dirac index has better properties than Dirac cohomology. For instance, it preserves short exact sequences. An algorithm for efficiently computing the Dirac index was given in \cite{ALTV} and implemented in \texttt{atlas} by Adams.

But Dirac index may carry less information than Dirac cohomology. For instance, an irreducible unitary representation $\pi$ for the linear split $F_4$ has been reported in Example 6.3 of \cite{DDY}: it has non-zero Dirac cohomology, but ${\rm DI}(\pi)$ vanishes. Indeed, both $H_D^+(\pi)$ and $H_D^-(\pi)$ consist of the same $\widetilde{K}$-type. Thus they cancel with each other when passing to Dirac index.  This kind of \emph{cancellation phenomenon} continues to happen for other higher rank Lie groups. All the Dirac series that we have known so far suggest that, in general, whenever $H_D^+(\pi)$ and $H_D^-(\pi)$ have overlap, then one would have ${\rm DI}(\pi)=0$. The underlying reason is that there should be certain dichotomy among the spin LKTs of the Dirac series $\pi$. This is the content of Conjecture 1.4 of \cite{D21}.

Now for the group \texttt{E8\_q}, we will meet more such examples.

\begin{example}\label{exam-minimal-repn}
Consider the third representation $\pi$ in Example \ref{exam-HJ}, namely, the minimal representation of $E_{8(-24)}$. It has infinitesimal character $[1, 1, 1, 0, 1, 1, 1, 1]$ and multiplicity-free $K$-types $\mu+n\beta$ for $n\geq 0$. Here $\mu=[0, 0, 0, 0, 0, 0, 0, 8]$. By Theorem \ref{thm-HP}, to compute $H_D(\pi)$, it suffices to look at the $K$-types of $\pi$ up to the height of
\begin{verbatim}
2*[1, 1, 1, 0, 1, 1, 1, 1]*rho_check(G)
Value: 970/1
\end{verbatim}
By computing the spin norm, we find that $\pi$ has ten spin LKTs: $\mu+n\beta$ for $1\leq n\leq 10$. All of them have spin norm $2\sqrt{95}$, which equals $\|\Lambda\|$. Thus $H_D(\pi)$ is non-zero. But ${\rm DI} (\pi)=0$. Indeed, both $\mu+\beta$ and $\mu+10\beta$ contribute the $\widetilde{K}$-type $[0, 0, 0, 0, 0, 0, 0, 18]$ to $H_D(\pi)$, but with opposite parities. As suggested by Conjecture 1.4 of \cite{D21}, a dichotomy among the spin LKTs are exhibited as follows:
$$
\mu+\beta\leftrightarrow \mu+10\beta, \mu+2\beta\leftrightarrow \mu+9\beta, \mu+3\beta\leftrightarrow \mu+8\beta, \mu+4\beta\leftrightarrow \mu+7\beta, \mu+5\beta\leftrightarrow \mu+6\beta.
$$
Here when two $K$-types are linked with an arrow, we mean that they contribute the same $\widetilde{K}$-type(s) to $H_D(\pi)$, but with opposite parities on each one. \hfill\qed
\end{example}

\begin{example}\label{exam-Table19}
Let us look at the fifth representation of Table \ref{table-EIX-11101011-part1}. It has infinitesimal character $[1, 1, 1, 0, 1, 0, 1, 1]$. This time, it suffices to look at its $K$-types up to the height of
\begin{verbatim}
2*[1, 1, 1, 0, 1, 0, 1, 1]*rho_check(G)
Value: 802/1
\end{verbatim}
Then by computing the spin norm, we find that there are eight spin LKTs in total. Among them, both $[5,1,0,0,0,0,0,5]$ and $[0,6,0,0,0,0,0,10]$ contribute the $\widetilde{K}$-type $[0, 0, 0, 0, 0, 0, 5, 5]$ to the Dirac cohomology, but
 with opposite signs. We write this as:
$$
[5,1,0,0,0,0,0,5]\leftrightarrow [0,6,0,0,0,0,0,10]: \, [0, 0, 0, 0, 0, 0, 5, 5].
$$
The story goes on as follows:
\begin{align*}
&[3,3,0,0,0,0,2,5] \leftrightarrow [2,2,2,0,0,0,0,10]: \, [0, 0, 0, 0, 0, 2, -1, 1];\\
&[4,1,1,0,0,0,0,7] \leftrightarrow [1,5,0,0,0,0,1,8]: \, [0, 0, 0, 0, 0, 1, 3, 4], [0, 0, 0, 0, 0, 0, 2, 4];\\
&[3,1,2,0,0,0,0,9] \leftrightarrow [2,4,0,0,0,0,2,6]: \, [\frac{1}{2},\frac{1}{2},\frac{1}{2},\frac{1}{2},\frac{1}{2},\frac{1}{2},\frac{1}{2},\frac{5}{2}],
[0, 0, 0, 0, 1, 1, 0, 2].
\end{align*}
Again, this representation has zero Dirac index. Note that the Dirac index can be computed directly in \texttt{atlas} using the algorithm of \cite{MPVZ}. For instance, here we have
\begin{verbatim}
G:E8_q
set p=parameter(KGB(G)[54724],[1,2,2,-1,2,-5,8,1]/1,[3,3,3,-3,3,-16,20,0]/2)
show_dirac_index(p)
Dirac index is 0
\end{verbatim}\hfill\qed
\end{example}
\begin{rmk}\label{rmk-coho-DI}
It is quite efficient to compute Dirac index. Indeed, it took a few seconds in Example \ref{exam-Table19}. On the other hand, branching the $K$-types of the representation up to the required height is much harder. For instance, the first representation of Table \ref{table-EIX-10010111-part1} cost 75 days, and this was not the hardest one.
\end{rmk}

It is interesting to note that the cancellation phenomenon is most likely to happen for the infinitesimal character $\rho_c$, which is conjugate to $[1, 0, 0, 1, 0, 1, 1, 1]$ under the action of $W(\frg, \frt_f)$. See Tables \ref{table-EIX-10010111-part1} and \ref{table-EIX-10010111-part2}. Indeed, the set $\Pi_{\rm FS}^{\rm u}(\rho_c)$ has $29$ FS-scattered representations in total, and cancellation happens for $12$ of them.

\begin{example}\label{exam-rhoc}
Let us look at the fifth representation $\pi$ of Table \ref{table-EIX-10010111-part1}. It has four spin LKTs, both $H_D^+(\pi)$ and $H_D^-(\pi)$ consist of two copies of the trivial $\widetilde{K}$-type. The dichotomy among the spin LKTs is demonstrated as follows:
$$
[1,0,3,0,0,0,1,11]\leftrightarrow [0,0,4,0,0,0,0,10], \quad [1,3,0,0,0,0,4,5]\leftrightarrow [0,3,1,0,0,0,3,4].
$$\hfill\qed
\end{example}

Now let us show that many FS-scattered members can be realized ``string limits".

\begin{example}\label{exam-string-limits}
Let us start with a $\theta$-stable parabolic subalgebra of $G$.
\begin{verbatim}
G:E8_q
set kgp=KGP(G, [0,1,2,3,4,5,6])
set P=kgp[3]
set L=Levi(P)
L
Value: connected real group with Lie algebra 'e7(so(12).su(2)).u(1)'
\end{verbatim}
Now let us take the trivial module \texttt{t} of $L$, and minus the \texttt{lambda} parameter of \texttt{t} by $[0,0,0,0,0,0,0,i]$. By doing cohomological induction via the following command
\begin{verbatim}
theta_induce_irreducible(parameter(x(t),lambda(t)-[0,0,0,0,0,0,0,i],nu(t)),G)
\end{verbatim}
we  get the module $A_{\frq}([0,0,0,0,0,0,0, -i])$, which is good for $i=0$, weakly good for $i=1$, and fair for $2\leq i\leq 13$. We present their \texttt{atlas} parameters as follows:
\begin{table}[H]
\centering
%\caption{Infinitesimal character $[0,1,1,0,1,0,1,1]$}
\begin{tabular}{lcr}
$i$ & \texttt{atlas} parameters & height  \\
\hline
$0$ & $(53504, [1,1,1,1,1,1,1,1], [1,0,1,2,0,2,0,-13])$ & $868$ \\
$1$ & $(53504, [1,1,1,1,1,1,1,0], [1,0,1,2,0,2,0,-13])$ & $810$ \\
$2$ & $(55211, [1,1,1,1,1,2,-2,3], [1,0,1,2,0,2,-13,13])$ & $752$ \\
$3$ & $(56790, [1,1,1,1,1,-1,3,1], [1,0,1,2,0,-11,13,0])$ & $694$ \\
$4$ & $(58238, [1,1,1,2,-4,5,2,1], [1,0,1,2,-11,11,2,0])$ & $636$ \\
$5$ & $(59541, [1,1,1,-2,4,1,1,1], [1,0,1,-9,11,0,2,0])$ & $578$ \\
$6$ & $(60705, [1,-5,-3,6,1,1,2,1], [1,-9,-8,9,2,0,2,0])$ & $520$ \\
$7$ & $(62674, [-3,8,6,-5,1,1,2,1], [-7,9,8,-8,2,0,2,0])$ & $462$ \\
$8$ & $(64228, [6,1,-5,8,-5,1,2,1], [7,1,-7,8,-6,0,2,0])$ & $404$ \\
$9$ & $(65296, [1,1,6,-3,6,-5,2,1], [0,1,7,-5,6,-6,2,0])$  & $346$ \\
$10$ & $(66049, [1,-2,1,5,-4,7,-4,1], [0,-4,2,5,-5,6,-4,0])$ & $288$ \\
$11$ & $(66733, [1,4,1,-3,6,-2,5,-4], [0,4,2,-4,5,-3,4,-4])$ & $230$ \\
$12$ & $(67007, [1,1,-2,5,-2,3,-2,5], [0,0,-2,4,-2,3,-3,4])$ & $172$ \\
$13$ & $(66950, [-3,1,4,0,3,-2,1,5], [-3,1,4,-1,2,-2,0,4])$ & $172$
\end{tabular}
\label{}
\end{table}
Here the last column records the height of the LKT of the representation.
We note that for $2\leq i\leq 13$, the modules are all FS-scattered Dirac series. In other words, they can be viewed as \emph{limits} of the following string:
\begin{verbatim}
(53504, [1,1,1,1,1,1,1,a], [1,0,1,2,0,2,0,-13])
\end{verbatim}
where $a$ runs over the non-negative integers. It is interesting to note that, up to signs, the Dirac indices of the last four $A_{\frq}(\lambda)$ modules are three copies of the following $\widetilde{K}$-modules
$$
[0,0,0,0,0,0,0,8], \, [0,0,0,0,0,0,0,6], \, [0,0,0,0,0,0,0,4], \, [0,0,0,0,0,0,0,2],
$$
respectively. The module $A_{\frq}([0,0,0,0,0,0,0, -i])$ has GK dimension $57$ for $0\leq i\leq 14$.
\hfill\qed
\end{example}

\section{Appendix}\label{sec-appendix}

This appendix presents all the $211$ FS-scattered Dirac series representations of $E_{8(-24)}$ according to their infinitesimal characters.  Those special unipotent representations in the sense of \cite{BV} are marked with $\clubsuit$. They are available from the list \cite{LSU}.

Whenever the cancellation phenomenon (see Section \ref{sec-examples}) happens, we will mark the representation $\pi$ with $*$.  There are fifteen such representations in total. We have checked that ${\rm DI}(\pi)=0$ for every such $\pi$. This gives further evidence to Conjecture 1.4 of \cite{D21}.

We will show a spin LKT in the \textbf{bold} fashion if it is a LKT at the same time.

\begin{table}[H]
\centering
\caption{Infinitesimal character $[0,1,1,0,1,0,1,1]$}
\begin{tabular}{lcc}
$\# x$ & $\lambda/\nu$ &  Spin LKTs   \\
\hline
$66950$ & $[-3,1,4,0,3,-2,1,5]$  & $[6,0,0,0,0,0,0,14]$, $[2,0,0,0,0,4,0,6]$,\\
&  $[-3,1,4,-1,2,-2,0,4]$& $[2,0,0,0,0,0,8,2]$ \\

$66125$ & $[-5,1,6,-1,4,-2,1,2]$  & $[0,0,0,0,0,1,0,24]$, $[4,0,0,0,0,1,0,20]$,\\
&  $[-\frac{11}{2},0,\frac{11}{2},-\frac{3}{2},\frac{7}{2},-2,0,\frac{3}{2}]$& $[0,0,0,0,0,0,10,2]$, $[0,0,0,0,0,5,0,12]$  \\

$45485$ & $[-3,4,5,-2,1,-5,7,2]$  & $[0,0,0,3,0,0,0,4]$, $[0,6,0,0,0,0,0,2]$\\
&  $[-3,4,4,-3,1,-6,7,1]$&  \\

$38214$ & $[-2,3,5,-2,1,-4,5,3]$  & $[0,0,0,0,0,1,0,28]$\\
&  $[-3,\frac{5}{2},\frac{11}{2},-\frac{5}{2},0,-\frac{11}{2},\frac{11}{2},\frac{5}{2}]$&  \\

$38212$ & $[-2,3,5,-2,1,-4,5,3]$  & $[5,0,0,0,0,0,1,19]$, $[1,0,0,0,0,4,1,11]$\\
&  $[-3,\frac{5}{2},\frac{11}{2},-\frac{5}{2},0,-\frac{11}{2},\frac{11}{2},\frac{5}{2}]$&  \\

$38210$ & $[-2,3,5,-2,1,-4,5,3]$  & $[5,1,0,0,0,0,0,13]$, $[2,1,0,0,0,3,0,11]$\\
&  $[-3,\frac{5}{2},\frac{11}{2},-\frac{5}{2},0,-\frac{11}{2},\frac{11}{2},\frac{5}{2}]$&  \\

$32933$ & $[-4,1,5,-4,5,0,1,1]$  & $\textbf{[0,0,0,0,0,0,0,30]}$\\
&  $[-6,0,6,-6,6,-1,1,1]$&  \\
$24579$ & $[-3,2,5,-4,6,-4,2,2]$  & $[1,0,0,0,2,1,0,12]$, $[1,0,1,0,1,1,1,14]$,\\
&  $[-4,\frac{3}{2},5,-5,6,-\frac{9}{2},1,1]$& $[2,0,0,1,0,2,0,12]$, $[0,0,0,3,0,0,0,8]$  \\
$18458$ & $[0,1,2,-1,2,-1,1,2]$  & $[0,1,0,0,1,0,6,2]$, $[1,0,0,1,0,0,6,0]$,\\
&  $[-\frac{5}{2},0,\frac{7}{2},-\frac{7}{2},6,-6,\frac{5}{2},1]$&  $[0,0,2,0,0,0,6,2]$\\
$18457$ & $[0,1,2,-1,2,-1,1,2]$  & $\textbf{[0,0,2,0,0,0,0,24]}$\\
&  $[-\frac{5}{2},0,\frac{7}{2},-\frac{7}{2},6,-6,\frac{5}{2},1]$&  \\

$6718$ & $[-1,2,6,-4,2,-1,1,3]$  & $[0,3,1,0,0,0,3,6]$, $[0,2,0,1,0,1,2,4]$\\
&  $[-3, \frac{3}{2},4,-3,1,-1,0,\frac{5}{2}]$&  \\
$5412$ & $[0,2,3,-2,1,0,1,2]$  & $[1,1,0,0,0,3,1,12]$, $[0,0,1,0,0,3,2,10]$\\
&  $[-2,\frac{5}{2},3,-3,0,-\frac{1}{2},1,2]$& 	
\end{tabular}
\label{table-EIX-01101011}
\end{table}

\begin{table}[H]
\centering
\caption{Infinitesimal character $[0,1,1,0,1,1,0,1]$}
\begin{tabular}{lcc}
$\# x$ & $\lambda/\nu$ &  Spin LKTs   \\
\hline
$43263$ & $[-3,5,6,-3,1,2,-7,10]$  & $[0,1,1,2,0,0,0,3]$, $[0,5,1,0,0,0,0,3]$\\
&  $[-3,4,4,-3,1,1,-7,8]$&  \\
$35758$ & $[-2,1,3,0,1,1,-2,3]$  & $[0,0,0,0,0,1,1,29]$\\
&  $[-3,\frac{5}{2},\frac{11}{2},-\frac{5}{2},0,0,-\frac{11}{2},8]$&  \\
$35756$ & $[-2,1,3,0,1,1,-2,3]$  & $[5,0,0,0,0,0,2,20]$, $[1,0,0,0,0,4,2,12]$\\
&  $[-3,\frac{5}{2},\frac{11}{2},-\frac{5}{2},0,0,-\frac{11}{2},8]$&  \\
$35754$ & $[-2,1,3,0,1,1,-2,3]$  & $[5,1,0,0,0,0,1,12]$, $[3,1,0,0,0,2,1,12]$\\
&  $[-3,\frac{5}{2},\frac{11}{2},-\frac{5}{2},0,0,-\frac{11}{2},8]$&  \\

$8173$ & $[-1,2,6,-4,1,2,-1,3]$  & $[0,3,1,0,0,1,2,7]$, $[0,2,0,1,0,2,1,5]$\\
&  $[-\frac{7}{2},1,\frac{9}{2},-3,0,\frac{3}{2},-\frac{3}{2},\frac{5}{2}]$&  \\
$4749$ & $[0,1,1,0,1,1,0,1]$  & $[1,1,0,0,0,3,2,13]$, $[0,0,1,0,0,3,3,11]$\\
&  $[-2,\frac{5}{2},3,-3,0,0,0,\frac{5}{2}]$& 	
\end{tabular}
\label{table-EIX-01101101}
\end{table}

\begin{table}[H]
\centering
\caption{Infinitesimal character $[0, 1, 1, 0, 1, 1, 1, 1]$}
\begin{tabular}{lcc}
$\# x$ & $\lambda/\nu$ &  Spin LKTs   \\
\hline
$62674$ & $[-3,8,6,-5,1,1,2,1]$  & $[6,0,0,0,0,0,0,2]$, $[5,0,0,0,0,0,4,4]$,\\
&  $[-7,9,8,-8,2,0,2,0]$& $[5,0,0,0,0,1,3,3]$, $[5,0,0,0,0,0,5,5]$, \\
& & $[5,0,0,0,0,1,4,4]$\\
$57322$ & $[-1,5,5,-4,1,1,2,1]$  & $[0,0,0,0,0,1,3,21]$, $[1,0,0,0,0,0,4,22]$,\\
&  $[-\frac{9}{2},\frac{17}{2},\frac{17}{2},-\frac{17}{2},0,0,4,1]$& $[0,0,0,0,0,2,2,22]$, $[1,0,0,0,0,1,3,23]$, \\
& & $[0,0,0,0,0,0,10,14]$, $[0,0,0,0,0,5,0,24]$\\
$9792$ & $[-1,1,4,-2,1,1,2,1]$  & $[0,2,0,0,0,5,0,10]$, $[1,3,0,0,0,3,1,10]$,\\
&  $[-\frac{9}{2},1,\frac{11}{2},-\frac{7}{2},0,0,2,1]$&  $[0,2,1,0,0,4,0,8]$\\
$4752$ & $[0,2,3,-2,1,1,1,2]$  & $[4,1,0,0,0,0,1,30]$, $[4,0,0,0,0,1,0,32]$\\
&  $[-\frac{5}{2},3,\frac{7}{2},-\frac{7}{2},0,0,0,3]$& 	
\end{tabular}
\label{table-EIX-01101111}
\end{table}

\begin{table}[H]
\centering
\caption{Infinitesimal character $[1,0,0,1,0,1,1,1]$, part I}
\begin{tabular}{lcc}
$\# x$ & $\lambda/\nu$ &  Spin LKTs   \\
\hline
$66865_{\clubsuit}$ & $[6,1,-1,2,-1,1,1,3]$  & $[5,0,0,0,0,0,1,17]$, $[1,0,0,0,0,4,1,9]$,\\
&  $[4,0,0,1,-1,0,0,4]$&  $[1,0,0,0,0,0,9,1]$\\
$66865_{\clubsuit}$ & $[6,1,-1,2,-1,1,1,4]$  & $[6,0,0,0,0,0,0,16]$, $[2,0,0,0,0,4,0,8]$,\\
&  $[4,0,0,1,-1,0,0,4]$&  $[2,0,0,0,0,0,8,0]$\\
$66818_{\clubsuit}$ & $[5,0,0,3,-2,1,1,3]$  & $[0,0,0,0,0,0,1,27]$, $[5,0,0,0,0,0,1,17]$,\\
&  $[4,0,0,\frac{5}{2},-\frac{5}{2},0,0,\frac{5}{2}]$& $[1,0,0,0,0,4,1,9]$, $[1,0,0,0,0,0,9,1]$ \\
$66492_{\clubsuit}$ & $[6,-1,-1,3,-1,1,1,2]$  & $[0,0,0,0,0,1,0,26]$, $[4,0,0,0,0,1,0,18]$,\\
&  $[6,-2,-2,3,-1,0,0,2]$& $[0,0,0,0,0,5,0,10]$, $[0,0,0,0,0,0,10,0]$\\
$58826^*$ & $[2,-6,-6,9,-2,3,2,1]$  & $[1,0,3,0,0,0,1,11]$, $[0,0,4,0,0,0,0,10]$,\\
&  $[\frac{3}{2},-\frac{13}{2},-\frac{13}{2},8,-\frac{3}{2},\frac{3}{2},2,0]$ & $[1,3,0,0,0,0,4,5]$, $[0,3,1,0,0,0,3,4]$	\\
$49177_{\clubsuit}$ & $[1,-1,-1,8,-6,1,1,2]$  & $[0,0,0,0,0,0,1,27]$\\
&  $[0,-2,-2,9,-7,0,0,2]$ &	\\
$47337^*$ & $[4,-5,-5,8,-2,1,1,3]$  & $[0,0,0,0,0,0,1,27]$, $[0,0,0,0,0,1,0,26]$\\
&  $[4,-\frac{13}{2},-\frac{13}{2},9,-\frac{5}{2},0,0, \frac{5}{2}]$ &	\\
$47335^*$ & $[4,-5,-5,8,-2,1,1,3]$  & $[0,0,0,0,0,5,0,10]$, $[1,0,0,0,0,4,1,9]$\\
&  $[4,-\frac{13}{2},-\frac{13}{2},9,-\frac{5}{2},0,0, \frac{5}{2}]$ &	\\
$43271$ & $[2,-3,-5,6,-5,8,1,2]$  & $[0,0,0,3,0,0,0,6]$, $[0,6,0,0,0,0,0,0]$\\
&  $[1,-4,-4,5,-5,6,1,1]$ &	\\
$40632_{\clubsuit}$ & $[5,-4,-4,5,1,0,1,0]$  & $\textbf{[0,0,0,0,0,0,0,28]}$\\
&  $[\frac{13}{2}, -\frac{13}{2}, -\frac{13}{2}, \frac{13}{2}, -\frac{3}{2}, \frac{3}{2}, 0, \frac{3}{2}]$&	
\end{tabular}
\label{table-EIX-10010111-part1}
\end{table}

\begin{table}[H]
\centering
\caption{Infinitesimal character $[1,0,0,1,0,1,1,1]$, part II}
\begin{tabular}{lcc}
$\# x$ & $\lambda/\nu$ &  Spin LKTs   \\
\hline
$37112$ & $[3,-3,-5,6,-3,4,1,3]$  & $[6,0,0,0,0,0,0,16]$, $[4,1,0,0,0,0,1,16]$,\\
&  $[\frac{5}{2},-3,-6,6,-5,5,0,\frac{5}{2}]$& $[1,1,0,0,0,3,1,10]$, $[2,0,0,0,0,4,0,8]$ \\

$35831$ & $[1,0,-3,4,-3,4,1,1]$  & $\textbf{[0,0,0,0,0,1,0,26]}$\\
&  $[\frac{5}{2},-\frac{5}{2},-\frac{11}{2},\frac{11}{2},-\frac{11}{2},\frac{11}{2},0,\frac{5}{2}]$& \\

$35827$ & $[1,0,-3,4,-3,4,1,1]$  & $[5,0,0,0,0,0,1,17]$, $[1,0,0,0,0,4,1,9]$\\
&  $[\frac{5}{2},-\frac{5}{2},-\frac{11}{2},\frac{11}{2},-\frac{11}{2},\frac{11}{2},0,\frac{5}{2}]$& \\

$35826$ & $[1,0,-3,4,-3,4,1,1]$  & $[5,1,0,0,0,0,0,15]$, $[2,1,0,0,0,3,0,9]$\\
&  $[\frac{5}{2},-\frac{5}{2},-\frac{11}{2},\frac{11}{2},-\frac{11}{2},\frac{11}{2},0,\frac{5}{2}]$& \\

$32165^*$ & $[1,-6,0,7,-3,1,1,1]$  & $[0,0,4,0,0,0,0,10]$, $[0,0,3,0,1,0,0,9]$\\
&  $[1,-7,-\frac{1}{2},7,-\frac{9}{2},0,2,0]$& \\

$31389^*$ & $[1,-5,-1,8,-4,1,1,1]$  & $[0,3,0,0,0,0,0,21]$, $[0,2,1,0,0,0,0,20]$\\
&  $[0,-\frac{13}{2},-\frac{1}{2},\frac{15}{2},-5,0,2,0]$& \\

$28777^*$ & $[2,-5,-1,7,-6,4,1,3]$  & $[1,1,0,0,1,1,3,5]$, $[1,2,0,0,0,1,4,4]$\\
&  $[1,-\frac{11}{2},-1,\frac{13}{2},-\frac{13}{2},3,1,1]$& \\

$27218^*$ & $[3,-4,-2,7,-6,3,2,3]$  & $[0,1,0,0,1,1,4,6]$, $[0,2,0,0,0,1,5,5]$,\\
&  $[2,-\frac{9}{2},-2,\frac{13}{2},-\frac{13}{2},2,2,1]$& $[1,0,0,0,2,0,4,4]$, $[1,1,0,0,1,0,5,3]$ \\

$19360^*$ & $[1,-5,-1,6,-4,2,3,3]$  & $[0,0,0,0,2,0,5,5]$, $[0,1,0,0,1,0,6,4]$\\
&  $[0,-\frac{11}{2},-1,\frac{11}{2},-\frac{9}{2},1,3,1]$&  \\

$17768^*$ & $[1,-1,-3,6,-4,2,2,1]$  & $[0,0,0,0,3,0,0,15]$, $[1,0,0,0,2,1,0,14]$\\
&  $[1,-1,-4,6,-5,1,1,1]$&  \\

$15697$ & $[6,-3,-4,5,-3,2,2,2]$  & $\textbf{[0,0,0,2,0,0,0,18]}$\\
&  $[5,-3,-4,4,-3,1,1,1]$&  \\

$14591^*$ & $[1,0,-2,5,-4,1,3,1]$  & $[2,0,0,0,0,4,0,8]$, $[1,0,1,0,0,3,1,7]$\\
&  $[1,0,-3,5,-5,0,2,1]$&  \\

$13658$ & $[3,-2,-1,4,-3,1,2,2]$  & $[0,0,0,0,0,5,0,10]$, $[1,1,0,0,0,3,1,10]$,\\
&  $[\frac{9}{2}, -\frac{5}{2}, -\frac{5}{2}, \frac{7}{2}, -\frac{7}{2}, 0, 2, 1]$& $[1,1,1,0,0,2,1,8]$, $[0,0,2,0,0,3,0,6]$ \\

$12664$ & $[2,0,-1,2,-1,1,1,2]$  & $[2,0,0,0,0,0,8,0]$, $[0,0,1,0,0,0,8,2]$\\
&  $[5,-2,-3,3,-3,0,2,1]$&  \\

$12662$ & $[2,0,-1,2,-1,1,1,2]$  & $[1,0,0,0,0,4,1,9]$, $[1,0,1,0,0,3,1,7]$\\
&  $[5,-2,-3,3,-3,0,2,1]$&  \\

$12660$ & $[2,0,-1,2,-1,1,1,2]$  & $[0,1,0,0,0,4,0,11]$, $[0,1,2,0,0,2,0,7]$\\
&  $[5,-2,-3,3,-3,0,2,1]$&  \\

$10759^*$ & $[3,-1,-3,4,-2,1,2,2]$  & $[1,0,0,0,0,0,9,1]$, $[0,0,1,0,0,0,8,2]$\\
&  $[3, -\frac{3}{2}, -\frac{9}{2}, \frac{9}{2},  -3, 0, \frac{3}{2}, \frac{3}{2}]$&  \\

$7482^*$ & $[2,-1,-1,4,-3,1,1,3]$  & $[0,3,1,0,0,0,3,4]$, $[0,3,0,1,0,0,2,3]$\\
&  $[\frac{5}{2}, -\frac{1}{2}, -\frac{5}{2}, \frac{7}{2}, -\frac{7}{2}, 0, 0, 3]$&  \\

$4287$ & $[2,-1,-2,4,-2,1,1,2]$  & $[0,3,0,0,0,0,5,6]$, $[0,1,0,0,1,1,4,6]$,\\
&  $[1, -\frac{5}{2}, -\frac{5}{2}, \frac{7}{2}, -\frac{3}{2}, 0, 0, 2 ]$& $[1,2,0,0,1,0,3,6]$, $[1,2,0,0,0,1,4,4]$	
\end{tabular}
\label{table-EIX-10010111-part2}
\end{table}

\begin{table}[H]
\centering
\caption{Infinitesimal character $[1, 0, 0, 1, 1, 0, 1, 1]$}
\begin{tabular}{lcc}
$\# x$ & $\lambda/\nu$ &  Spin LKTs   \\
\hline
$40971$ & $[1,-2,-3,5,1,-4,6,1]$  & $[0,1,1,2,0,0,0,5]$, $[0,5,1,0,0,0,0,1]$\\
&  $[1,-4,-4,5,1,-6,7,1]$&  \\
$33341$ & $[2,-1,-3,4,1,-3,4,2]$  & $\textbf{[0,0,0,0,0,1,1,27]}$\\
&  $[\frac{5}{2},-\frac{5}{2},-\frac{11}{2},\frac{11}{2},0,-\frac{11}{2},\frac{11}{2},\frac{5}{2}]$&  \\
$33337$ & $[2,-1,-3,4,1,-3,4,2]$  & $[5,0,0,0,0,0,2,18]$, $[1,0,0,0,0,4,2,10]$\\
&  $[\frac{5}{2},-\frac{5}{2},-\frac{11}{2},\frac{11}{2},0,-\frac{11}{2},\frac{11}{2},\frac{5}{2}]$&  \\
$33336$ & $[2,-1,-3,4,1,-3,4,2]$  & $[5,1,0,0,0,0,1,14]$, $[3,1,0,0,0,2,1,10]$\\
&  $[\frac{5}{2},-\frac{5}{2},-\frac{11}{2},\frac{11}{2},0,-\frac{11}{2},\frac{11}{2},\frac{5}{2}]$& 	
\end{tabular}
\label{table-EIX-10011011}
\end{table}

\begin{table}[H]
\centering
\caption{Infinitesimal character $[1, 0, 0, 1, 1, 1, 0, 1]$}
\begin{tabular}{lcc}
$\# x$ & $\lambda/\nu$ &  Spin LKTs   \\
\hline
$38606$ & $[2,-3,-4,6,1,2,-6,8]$  & $[0,2,2,1,0,0,0,4]$, $[0,4,2,0,0,0,0,2]$\\
&  $[1,-4,-4,5,1,1,-7,8]$&  \\
$30869$ & $[1,0,-2,3,1,1,-2,3]$  & $\textbf{[0,0,0,0,0,1,2,28]}$\\
&  $[\frac{5}{2},-\frac{5}{2},-\frac{11}{2},\frac{11}{2},0,0,-\frac{11}{2},8]$&  \\
$30865$ & $[1,0,-2,3,1,1,-2,3]$  & $[5,0,0,0,0,0,3,19]$, $[1,0,0,0,0,4,3,11]$\\
&  $[\frac{5}{2},-\frac{5}{2},-\frac{11}{2},\frac{11}{2},0,0,-\frac{11}{2},8]$&  \\
$30864$ & $[1,0,-2,3,1,1,-2,3]$  & $[5,1,0,0,0,0,2,13]$, $[4,1,0,0,0,1,2,11]$\\
&  $[\frac{5}{2},-\frac{5}{2},-\frac{11}{2},\frac{11}{2},0,0,-\frac{11}{2},8]$& 	
\end{tabular}
\label{table-EIX-10011101}
\end{table}

\begin{table}[H]
\centering
\caption{Infinitesimal character $[1, 0, 0, 1, 1, 1, 1, 1]$}
\begin{tabular}{lcc}
$\# x$ & $\lambda/\nu$ &  Spin LKTs   \\
\hline
$60705$ & $[1,-5,-3,6,1,1,2,1]$  & $\textbf{[6,0,0,0,0,0,0,0]}$, $[6,0,0,0,0,0,4,4]$\\
&  $[1,-9,-8,9,2,0,2,0]$&  \\
$54212$ & $[3,-3,-3,4,1,1,2,1]$  & $[0,0,0,0,0,1,4,22]$, $[0,0,0,0,0,0,10,16]$, \\
&  $[4,-\frac{17}{2},-\frac{17}{2},\frac{17}{2},0,0,4,1]$&  $[0,0,0,0,0,5,0,26]$	
\end{tabular}
\label{table-EIX-10011111}
\end{table}

\begin{table}[H]
\centering
\caption{Infinitesimal character $[1,0,1,1,0,1,0,1]$}
\begin{tabular}{lcc}
$\# x$ & $\lambda/\nu$ &  Spin LKTs   \\
\hline
$66049$ & $[1,-2,1,5,-4,7,-4,1]$  & $[6,0,0,0,0,0,0,8]$, $[2,0,0,0,0,4,0,0]$, \\
&  $[0,-4,2,5,-5,6,-4,0]$& $[2,0,0,0,0,0,8,8]$ \\

$63572$ & $[1,-3,1,4,-3,7,-3,1]$  & $[0,0,0,0,0,1,0,18]$, $[0,0,0,0,0,0,10,8]$ \\
&  $[0,-\frac{9}{2},0,\frac{9}{2},-\frac{9}{2},\frac{17}{2},-\frac{9}{2},1]$& $[0,0,0,0,0,5,0,18]$, $[4,0,0,0,0,1,0,26]$  \\

$26300$ & $[1,-4,1,5,-2,1,-2,5]$  & $\textbf{[2,0,0,0,0,4,0,0]}$\\
&  $[0,-5,2,5,-4,1,-4,5]$&  \\

$24970$ & $[1,-2,2,3,-2,2,-2,3]$  & $[0,6,0,0,0,0,0,8]$, $[0,0,0,3,0,0,0,14]$ \\
&  $[1,-\frac{9}{2},2,\frac{9}{2},-\frac{9}{2},2,-\frac{9}{2},\frac{9}{2}]$&  \\

$24315$ & $[1,-2,2,5,-4,2,-3,4]$  & $[0,1,0,0,0,4,0,19]$, $[3,1,0,0,0,1,0,25]$ \\
&  $[0,-4,2,5,-5,2,-\frac{9}{2},\frac{9}{2}]$&  \\	

$13993$ & $[1,-1,1,2,-1,3,-2,1]$  & $[0,0,0,0,0,0,1,35]$ \\
&  $[0,-\frac{9}{2},0,\frac{9}{2},-\frac{9}{2},4,-4,5]$&  \\

$13990$ & $[1,-1,1,2,-1,3,-2,1]$  & $[1,0,0,0,0,0,9,9]$\\
&  $[0,-\frac{9}{2},0,\frac{9}{2},-\frac{9}{2},4,-4,5]$&  \\

$12936$ & $[1,-2,1,3,-2,3,-2,3]$  & $\textbf{[0,0,0,0,0,0,0,36]}$ \\
&  $[0,-\frac{9}{2},0,\frac{9}{2},-\frac{9}{2},\frac{9}{2},-\frac{9}{2},\frac{9}{2}]$&  \\	
\end{tabular}
\label{table-EIX-10110101}
\end{table}

\begin{table}[H]
\centering
\caption{Infinitesimal character $[1,0,1,1,0,1,1,1]$}
\begin{tabular}{lcc}
$\# x$ & $\lambda/\nu$ &  Spin LKTs   \\
\hline

$37303$ & $[2,-3,1,5,-3,1,1,1]$  & $[0,4,1,0,0,0,0,18]$, $[0,5,0,0,0,0,3,14]$,\\
&  $[\frac{3}{2},-\frac{15}{2},0,9,-7,0,2,0]$&  $[0,2,3,0,0,0,0,20]$ \\

$33662$ & $[1,-5,1,6,-3,1,1,3]$  & $\textbf{[0,0,0,0,0,6,0,0]}$, $[0,0,0,0,2,4,0,2]$\\
&  $[0,-9,0,9,-6,1,1,1]$&  \\

$16077$ & $[1,-1,1,2,-1,1,3,-1]$  & $\textbf{[2,0,0,0,0,0,0,36]}$ \\
&  $[0,-\frac{11}{2},0,\frac{11}{2},-\frac{11}{2},0,5,1]$&  \\	
\end{tabular}
\label{table-EIX-10110111}
\end{table}

%\begin{table}[H]
%\centering
%\caption{Infinitesimal character $[1,0,1,1,0,1,1,1]$}
%\begin{tabular}{lcc}
%$\# x$ & $\lambda/\nu$ &  Spin LKTs   \\
%\hline
%	
%
%$37303$ & $[2,-3,1,5,-3,1,1,1]$  & $[0,4,1,0,0,0,0,18]$, $[0,5,0,0,0,0,3,14]$, \\
%&  $[\frac{3}{2},-\frac{15}{2},0,9,-7,0,2,0]$& $[0,2,3,0,0,0,0,20]$ \\
%
%$33662$ & $[1,-5,1,6,-3,1,1,3]$  & $[0,0,0,0,0,6,0,0]$, $[0,0,0,0,2,4,0,2]$\\
%&  $[0,-9,0,9,-6,1,1,1]$&  \\
%
%$16077$ & $[1,-1,1,2,-1,1,3,-1]$  & $[2,0,0,0,0,0,0,36]$ \\
%&  $[0,-\frac{11}{2},0,\frac{11}{2},-\frac{11}{2},0,5,1]$&  \\	
%\end{tabular}
%\label{table-EIX-11010101}
%\end{table}

\begin{table}[H]
\centering
\caption{Infinitesimal character $[1,1,0,1,0,1,0,1]$, Part I}
\begin{tabular}{lcc}
$\# x$ & $\lambda/\nu$ &  Spin LKTs   \\
\hline
$67007$ & $[1,1,-2,5,-2,3,-2,5]$  & $[6,0,0,0,0,0,0,12]$, $[2,0,0,0,0,4,0,4]$, \\
&  $[0,0,-2,4,-2,3,-3,4]$& $[2,0,0,0,0,0,8,4]$ \\

$65759$ & $[1,1,-4,5,-1,4,-2,2]$  & $[0,0,0,0,0,1,0,22]$, $[0,0,0,0,0,0,10,4]$, \\
&  $[0,0,-5,5,-1,4,-3,1]$&  $[4,0,0,0,0,1,0,22]$, $[0,0,0,0,0,5,0,14]$\\

$47680$ & $[4,2,-1,3,-2,2,-6,8]$  & $[0,0,0,3,0,0,0,2]$, $[0,6,0,0,0,0,0,4]$ \\
&  $[3,1,-2,3,-2,1,-7,8]$&  \\

$46793$ & $[2,2,-4,5,-4,6,-3,1]$  & $[0,2,1,0,0,0,0,16]$, $[0,0,1,0,2,0,0,20]$, \\
&  $[\frac{3}{2},2,-\frac{11}{2},\frac{11}{2},-\frac{11}{2},7,-5,0]$& $[0,0,3,0,0,0,2,16]$, $[0,3,0,0,0,0,5,10]$ \\

$43240$ & $[1,2,-7,8,-5,6,-3,3]$  & $[2,0,0,0,0,4,0,4]$, $[0,0,0,2,0,2,0,0]$, \\
&  $[0,3,-6,6,-5,5,-4,1]$&  $[0,0,0,0,4,0,0,4]$\\

$40735$ & $[3,1,0,3,-2,1,-4,7]$  & $[5,1,0,0,0,0,0,11]$, $[2,1,0,0,0,3,0,13]$ \\
&  $[3,0,0,\frac{5}{2},-\frac{5}{2},0,-\frac{11}{2},8]$& \\

$35410$ & $[1,1,-1,4,-2,1,-3,6]$  & $[0,0,0,0,0,1,0,30]$ \\
&  $[0,0,-2,5,-3,0,-5,8]$&   \\

$35408$ & $[1,1,-1,4,-2,1,-3,6]$  & $[5,0,0,0,0,0,1,21]$, $[1,0,0,0,0,4,1,13]$ \\
&  $[0,0,-2,5,-3,0,-5,8]$&  \\

$25275$ & $[1,1,-3,4,-3,4,1,-1]$  & $\textbf{[0,0,0,0,0,0,0,32]}$ \\
&  $[0,0,-\frac{11}{2},\frac{11}{2},-\frac{11}{2},\frac{11}{2},-\frac{1}{2},1]$& \\

$17087$ & $[7,2,-5,1,-1,4,-3,8]$  & $[0,0,2,0,0,3,0,2]$ \\
&  $[5,1,-4,1,-2,3,-3,4]$&  \\

$15858$ & $[2,2,-4,6,-4,2,-2,4]$  & $[0,2,0,0,2,1,0,6]$, $[0,0,2,0,0,3,0,10]$ \\
&  $[1,1,-4,5,-4,1,-2,3]$&  \\

$15898$ & $[6,3,-3,1,-3,5,-2,3]$  & $[1,0,2,0,0,1,1,17]$, $[1,0,3,0,0,0,1,15]$ \\
&  $[5,\frac{3}{2},-\frac{7}{2},0,-2,4,-3,3]$&\\

$14814$ & $[1,0,-1,1,1,2,-1,2]$  & $[0,3,0,0,0,0,0,25]$ \\
&  $[5,\frac{3}{2},-\frac{7}{2},0,-\frac{3}{2},\frac{7}{2},-\frac{7}{2},\frac{7}{2}]$&  \\

$14812$ & $[1,0,-1,1,1,2,-1,2]$  & $[0,3,0,0,2,0,0,7]$ \\
&  $[5,\frac{3}{2},-\frac{7}{2},0,-\frac{3}{2},\frac{7}{2},-\frac{7}{2},\frac{7}{2}]$&  \\

$14811$ & $[1,0,-1,1,1,2,-1,2]$  & $[0,0,3,0,0,0,2,16]$ \\
&  $[5,\frac{3}{2},-\frac{7}{2},0,-\frac{3}{2},\frac{7}{2},-\frac{7}{2},\frac{7}{2}]$&  \\

$14810$ & $[1,0,-1,1,1,2,-1,2]$  & $[2,0,2,0,0,1,0,16]$ \\
&  $[5,\frac{3}{2},-\frac{7}{2},0,-\frac{3}{2},\frac{7}{2},-\frac{7}{2},\frac{7}{2}]$&  \\

$12610$ & $[2,1,-1,2,-1,1,0,2]$  & $[1,3,0,0,0,0,4,1]$ \\
&  $[1,0,-3,5,-5,2,-2,3]$&  \\

$12609$ & $[2,1,-1,2,-1,1,0,2]$  & $\textbf{[0,0,2,0,0,0,0,26]}$ \\
&  $[1,0,-3,5,-5,2,-2,3]$&  \\

$12607$ & $[2,1,-1,2,-1,1,0,2]$  & $[2,0,0,0,2,0,1,17]$, $[0,0,0,2,0,0,3,13]$ \\
&  $[1,0,-3,5,-5,2,-2,3]$&  \\

$11607$ & $[4,1,-3,1,-1,4,-1,4]$  & $[2,0,0,0,0,0,8,4]$, $[0,0,1,0,0,0,8,6]$ \\
&  $[\frac{7}{2},0,-\frac{7}{2},0,-\frac{1}{2},\frac{7}{2},-\frac{5}{2},\frac{7}{2}]$&  \\

$10670$ & $[2,1,-1,1,0,3,-2,1]$  & $[1,0,1,0,0,0,7,5]$\\
&  $[\frac{7}{2},0,-\frac{7}{2},0,0,3,-3,4]$&
\end{tabular}
\label{table-EIX-11010101-part1}
\end{table}

\begin{table}[H]
\centering
\caption{Infinitesimal character $[1,1,0,1,0,1,0,1]$, Part II}
\begin{tabular}{lcc}
$\# x$ & $\lambda/\nu$ &  Spin LKTs   \\
\hline

$7444$ & $[2,1,-1,2,-1,1,0,2]$  & $[1,0,0,0,0,0,9,5]$ \\
&  $[3,0,-3,2,-2,0,0,4]$&  \\

$6464$ & $[3,1,-2,3,-2,1,0,3]$  & $\textbf{[0,0,0,0,0,0,10,4]}$ \\
&  $[3,0,-3,3,-3,0,0,3]$&  \\

$6455$ & $[3,1,-2,3,-2,1,0,3]$  & $[5,1,0,0,0,0,0,19]$ \\
&  $[3,0,-3,3,-3,0,0,3]$&  \\	

$6872$ & $[3,3,-2,2,-3,4,-2,2]$  & $[0,0,1,1,0,1,2,14]$, $[1,0,0,2,0,0,2,12]$ \\
&  $[\frac{5}{2},1,-\frac{5}{2},2,-\frac{7}{2},\frac{7}{2},-\frac{5}{2},1]$&  \\

$6813$ & $[4,1,-1,2,-3,4,-2,2]$  & $[3,1,0,0,0,1,0,21]$\\
&  $[3,0,-2,2,-\frac{7}{2},\frac{7}{2},-\frac{5}{2},1]$&  \\

$6773$ & $[3,1,-2,2,-2,5,-3,2]$  & $[1,0,1,0,0,0,7,5]$, $[0,0,0,1,0,0,7,7]$ \\
&  $[\frac{5}{2},0,-\frac{5}{2},2,-3,4,-3,1]$&
\end{tabular}
\label{table-EIX-11010101-part2}
\end{table}

\begin{table}[H]
\centering
\caption{Infinitesimal character $[1,1,0,1,0,1,1,1]$}
\begin{tabular}{lcc}
$\# x$ & $\lambda/\nu$ &  Spin LKTs   \\
\hline
$64228$ & $[6,1,-5,8,-5,1,2,1]$  & $[6,0,0,0,0,0,0,4]$, $[4,0,0,0,0,2,2,2]$, \\
&  $[7,1,-7,8,-6,0,2,0]$& $[4,0,0,0,0,1,4,4]$, $[4,0,0,0,0,0,6,6]$\\

$60124$ & $[3,1,-2,6,-5,1,3,1]$  &  $[0,0,0,0,0,0,10,12]$, $[0,0,0,0,0,5,0,22]$,\\
&  $[\frac{9}{2},0,-\frac{9}{2},\frac{17}{2},-\frac{17}{2},0,4,1]$& $[n,0,0,0,0,1,2,20+2n]$, $0\leq n\leq 2$ \\

$26325$ & $[5,1,-4,1,0,3,1,3]$  & $[2,0,0,0,0,4,2,2]$, $[2,1,0,0,0,3,3,4]$ \\
&  $[7,0,-7,2,-2,3,1,1]$&\\

$25050$ & $[4,1,-3,2,-1,2,2,1]$  & $\textbf{[0,4,0,0,0,0,1,21]}$,  \\
&  $[7,1,-7,2,-\frac{5}{2},\frac{5}{2},2,0]$& $[1,3,0,0,0,1,0,23]$, $[0,1,0,0,0,4,0,23]$\\

$24861$ & $[7,1,-4,2,-2,3,2,1]$  & $[0,6,0,0,0,0,0,12]$, $[1,2,3,0,0,0,0,10]$, \\
&  $[\frac{15}{2},0,-\frac{13}{2},2,-\frac{5}{2},\frac{5}{2},2,0]$&$[0,2,2,1,0,0,0,12]$\\

$20270$ & $[1,2,-3,5,-4,2,2,1]$  & $\textbf{[0,0,0,0,4,0,0,12]}$, $[1,0,0,0,3,1,0,13]$, \\
&  $[1,\frac{3}{2},-5,6,-6,\frac{3}{2},1,1]$&  $[2,0,0,0,2,2,0,14]$, $[0,0,0,2,2,0,0,10]$\\

$13517$ & $[3,1,-2,1,0,1,3,1]$  &  $[1,0,0,0,0,2,7,7]$, $[0,0,0,0,1,2,6,5]$\\
&  $[\frac{9}{2},0,-\frac{9}{2},0,-\frac{1}{2},1,\frac{7}{2},1]$&\\

$14593$ & $[1,1,-1,3,-2,1,2,1]$  &  $\textbf{[0,3,0,0,0,0,7,0]}$, $[0,2,1,0,0,0,7,1]$ \\
&  $[1,0,-\frac{7}{2},6,-6,0,\frac{5}{2},1]$&  \\

$14592$ & $[1,1,-1,3,-2,1,2,1]$  & $\textbf{[0,0,3,0,0,0,0,26]}$, $[0,0,3,0,0,0,2,24]$, \\
&  $[1,0,-\frac{7}{2},6,-6,0,\frac{5}{2},1]$&   $[0,0,1,0,2,0,0,28]$\\

$7447$ & $[1,1,0,2,-1,1,1,0]$  & $[0,0,0,0,0,0,1,39]$ \\
&  $[\frac{7}{2},0,-\frac{7}{2},2,-2,0,0,5]$&\\

$6473$ & $[2,1,-1,2,-1,1,1,2]$  & $\textbf{[0,0,0,0,0,0,0,40]}$ \\
&  $[\frac{7}{2},0,-\frac{7}{2},\frac{7}{2},-\frac{7}{2},0,0,\frac{7}{2}]$&
\end{tabular}
\label{table-EIX-11010111}
\end{table}

\begin{table}[H]
\centering
\caption{Infinitesimal character $[1,1,0,1,1,0,1,0]$}
\begin{tabular}{lcc}
$\# x$ & $\lambda/\nu$ &  Spin LKTs   \\
\hline

$11889$ & $[4,2,-2,1,2,-3,4,-1]$  & $[2,1,2,0,0,0,2,17]$, $[2,1,1,0,0,1,2,19]$ \\
&  $[5,\frac{3}{2},-\frac{7}{2},0,2,-4,4,-3]$&\\

$8207$ & $[3,1,-2,1,2,-2,5,-2]$  & $[2,0,0,0,0,1,8,6]$, $[0,0,1,0,0,1,8,8]$ \\
&  $[\frac{7}{2},0,-\frac{7}{2},0,3,-\frac{7}{2},\frac{9}{2},-\frac{7}{2}]$&
\end{tabular}
\label{table-EIX-11011010}
\end{table}

\begin{table}[H]
\centering
\caption{Infinitesimal character $[1,1,0,1,1,0,1,1]$}
\begin{tabular}{lcc}
$\# x$ & $\lambda/\nu$ &  Spin LKTs   \\
\hline
$23904$ & $[6,1,-5,2,1,-1,4,3]$  & $[2,0,0,0,0,5,1,1]$, $[3,0,0,0,0,4,2,2]$, \\
&  $[7,0,-7,2,1,-3,4,1]$& $[2,1,0,0,0,4,2,3]$\\

$22619$ & $[4,1,-3,2,1,-1,3,1]$  & $\textbf{[0,4,0,0,0,0,2,22]}$, $[1,4,0,0,0,0,1,23]$, \\
&  $[7,1,-7,2,0,-\frac{5}{2},\frac{9}{2},0]$& $[1,3,0,0,0,1,1,24]$, $[0,2,0,0,0,3,1,23]$\\

$22447$ & $[5,1,-2,1,1,-1,3,1]$  & $[1,6,0,0,0,0,0,12]$, $[2,3,2,0,0,0,0,11]$ \\
&  $[\frac{15}{2},0,-\frac{13}{2},2,0,-\frac{5}{2},\frac{9}{2},0]$& $[1,3,1,1,0,0,0,13]$, $[1,2,2,1,0,0,0,12]$\\

$12505$ & $[2,1,-1,1,1,0,3,-1]$  & $\textbf{[0,0,0,0,0,3,7,7]}$,  \\
&  $[\frac{9}{2},0,-\frac{9}{2},0,0,0,4,1]$& $[1,0,0,0,0,2,8,8]$, $[0,0,0,0,1,2,7,6]$ \\

$8964$ & $[1,1,0,1,2,-1,1,0]$  & $[1,0,0,0,0,0,1,39]$ \\
&  $[4,0,-4,0,3,-3,0,5]$&
\end{tabular}
\label{table-EIX-11011011}
\end{table}

\begin{table}[H]
\centering
\caption{Infinitesimal character $[1,1,0,1,1,1,0,1]$}
\begin{tabular}{lcc}
$\# x$ & $\lambda/\nu$ &  Spin LKTs   \\
\hline
$21544$ & $[5,1,-4,2,1,1,-1,6]$  & $\textbf{[2,0,0,0,0,6,0,0]}$, $[3,0,0,0,0,5,1,1]$, \\
&  $[7,0,-7,2,1,1,-4,5]$& $[2,1,0,0,0,5,1,2]$\\

$20259$ & $[4,1,-3,2,1,2,-2,3]$  & $\textbf{[0,4,0,0,0,0,3,23]}$, $[1,4,0,0,0,0,2,24]$, \\
&  $[7,1,-7,2,0,2,-\frac{9}{2},\frac{9}{2}]$& $[0,3,0,0,0,2,2,23]$, $[1,3,0,0,0,1,2,25]$\\

$20091$ & $[5,1,-2,1,1,2,-2,3]$  & $[2,6,0,0,0,0,0,12]$, $[3,4,1,0,0,0,0,12]$, \\
&  $[\frac{15}{2},0,-\frac{13}{2},2,0,2,-\frac{9}{2},\frac{9}{2}]$& $[2,4,0,1,0,0,0,14]$, $[2,3,1,1,0,0,0,13]$\\

$10675$ & $[1,1,0,1,1,2,-1,0]$  & $[2,0,0,0,0,0,1,39]$ \\
&  $[\frac{9}{2},0,-\frac{9}{2},0,0,4,-4,5]$&\\

$10672$ & $[1,1,0,1,1,2,-1,0]$  & $\textbf{[0,0,0,0,0,3,8,8]}$, $[1,0,0,0,0,2,9,9]$, \\
&  $[\frac{9}{2},0,-\frac{9}{2},0,0,4,-4,5]$& $[0,0,0,0,1,2,8,7]$
\end{tabular}
\label{table-EIX-11011101}
\end{table}

\begin{table}[H]
\centering
\caption{Infinitesimal character $[1,1,0,1,1,1,1,1]$}
\begin{tabular}{lcc}
$\# x$ & $\lambda/\nu$ &  Spin LKTs   \\
\hline
$32632$ & $[4,2,-2,1,1,1,1,1]$  & $\textbf{[0,7,0,0,0,0,0,19]}$, $[0,7,0,0,0,0,1,18]$,  \\
&  $[\frac{21}{2},\frac{3}{2},-9,0,2,0,2,0]$& $[0,6,1,0,0,0,0,20]$\\

$28770$ & $[4,1,-3,1,2,1,1,3]$  &  $\textbf{[0,0,0,0,0,8,0,0]}$, $[0,0,0,0,1,7,0,1]$\\
&  $[9,0,-9,0,3,1,1,1]$& \\

$12508$ & $[1,1,0,1,1,1,2,-1]$  &  $\textbf{[4,0,0,0,0,0,0,40]}$, $[3,1,0,0,0,0,0,41]$\\
&  $[\frac{11}{2},0,-\frac{11}{2},0,0,0,5,1]$&
\end{tabular}
\label{table-EIX-11011111}
\end{table}

\begin{table}[H]
\centering
\caption{Infinitesimal character $[1,1,1,0,1,0,1,0]$}
\begin{tabular}{lcc}
$\# x$ & $\lambda/\nu$ &  Spin LKTs   \\
\hline
$66733$ & $[1,4,1,-3,6,-2,5,-4]$  & $[6,0,0,0,0,0,0,10]$, $[2,0,0,0,0,4,0,2]$, \\
&  $[0,4,2,-4,5,-3,4,-4]$& $[2,0,0,0,0,0,8,6]$\\

$65437$ & $[1,5,1,-4,5,-1,5,-2]$  & $[0,0,0,0,0,1,0,20]$, $[0,0,0,0,0,0,10,6]$, \\
&  $[0,\frac{9}{2},0,-\frac{9}{2},\frac{9}{2},-\frac{1}{2},\frac{9}{2},-\frac{7}{2}]$& $[4,0,0,0,0,1,0,24]$, $[0,0,0,0,0,5,0,16]$ \\

$19236$ & $[1,4,1,-3,4,-3,4,0]$  & $\textbf{[0,0,0,0,0,0,0,34]}$ \\
&  $[0,5,0,-5,5,-5,5,0]$& \\

$15963$ & $[2,4,2,-1,1,-4,5,-2]$  & $[2,0,1,0,1,0,1,20]$, $[2,0,2,0,0,1,0,18]$ \\
&  $[1,4,1,-1,0,-\frac{9}{2},\frac{9}{2},-\frac{7}{2}]$& \\

$15362$ & $[1,4,1,-1,2,-4,6,-3]$  & $[1,0,1,0,0,0,7,7]$, $[0,0,0,1,0,0,7,9]$ \\
&  $[0,4,0,-1,1,-\frac{9}{2},\frac{11}{2},-\frac{9}{2}]$& \\

$15489$ & $[1,2,2,-3,4,-2,4,-2]$  & $[0,0,1,1,0,0,5,7]$, $[1,1,0,0,1,0,5,9]$, \\
&  $[0,\frac{3}{2},2,-\frac{9}{2},\frac{9}{2},-\frac{7}{2},5,-4]$& $[0,0,0,0,2,0,5,11]$ \\

$13885$ & $[2,2,3,-2,2,-3,4,-1]$  & $[1,0,3,0,0,0,1,17]$, $[1,0,2,0,0,1,1,19]$ \\
&  $[\frac{3}{2},\frac{3}{2},\frac{7}{2},-\frac{7}{2},2,-4,4,-3]$& \\

$13457$ & $[1,3,2,-3,4,-4,7,-4]$  & $[1,0,0,1,0,0,6,8]$, $[0,1,0,0,1,0,6,10]$ \\
&  $[0,\frac{3}{2},2,-\frac{7}{2},\frac{7}{2},-4,\frac{11}{2},-\frac{9}{2}]$& \\

$11745$ & $[1,1,3,-3,4,-2,3,-1]$  & $[2,0,0,1,0,1,0,22]$, $[2,0,1,0,1,0,1,20]$, \\
&  $[1,0,3,-\frac{9}{2},\frac{9}{2},-\frac{7}{2},\frac{7}{2},-\frac{5}{2}]$& $[3,0,0,0,2,0,0,18]$\\

$9954$ & $[3,1,3,-3,4,-4,5,-2]$  & $[3,0,0,1,0,0,1,21]$, $[3,0,1,0,1,0,0,19]$ \\
&  $[1,0,3,-\frac{7}{2},\frac{7}{2},-4,4,-3]$& \\

$9611$ & $[1,1,3,-2,2,-2,5,-2]$  & $[2,0,0,0,0,0,8,6]$, $[0,0,1,0,0,0,8,8]$ \\
&  $[0,0,\frac{7}{2},-\frac{7}{2},3,-\frac{7}{2},\frac{9}{2},-\frac{7}{2}]$& \\

$8436$ & $[2,4,2,-3,2,-2,4,-1]$  & $[1,1,0,0,1,1,3,11]$, $[0,2,0,0,1,0,4,13]$ \\
&  $[1,4,\frac{3}{2},-4,\frac{3}{2},-2,3,-2]$& \\

$4991$ & $[1,1,1,0,1,1,0,-1]$  & $[5,0,0,0,0,0,1,23]$, $[5,1,0,0,0,0,0,21]$ \\
&  $[\frac{5}{2},\frac{5}{2},0,-\frac{5}{2},0,-\frac{1}{2},3,-2]$&
\end{tabular}
\label{table-EIX-11101010}
\end{table}

\begin{table}[H]
\centering
\caption{Infinitesimal character $[1,1,1,0,1,0,1,1]$, part I}
\begin{tabular}{lcc}
$\# x$ & $\lambda/\nu$ &  Spin LKTs   \\
\hline
$67055_{\clubsuit}$ & $[1,2,1,-1,3,-1,4,1]$  & $[0,0,0,0,0,2,5,5]$,  $[0,0,0,0,0,3,5,5]$, \\
&  $[0,1,0,-1,4,-3,4,1]$&  $[0,0,0,0,0,n,10-2n,10-2n]$, $0\leq n\leq 5$\\

$65831_{\clubsuit}$ & $[1,6,1,-5,6,-3,7,0]$  & $[3,0,0,0,0,0,0,22]$, $[2,0,0,0,0,0,2,22]$,
 \\
&  $[0,\frac{9}{2},0,-\frac{9}{2}, \frac{9}{2}, -\frac{1}{2}, 1, \frac{9}{2}]$& $[n,0,0,0,0,0,1,2n+17]$, $0\leq n\leq 5$\\

$65296$ & $[1,1,6,-3,6,-5,2,1]$  & $[6,0,0,0,0,0,0,6]$, $[3,0,0,0,0,1,4,4]$, \\
&  $[0,1,7,-5,6,-6,2,0]$& $[3,0,0,0,0,2,4,4]$,  \\
& & $[3,0,0,0,0,n,7-2n,7-2n]$, $0\leq n\leq 3$\\

$61948$ & $[1,1,3,-2,6,-5,3,1]$  &  $[2,0,0,0,0,0,2,22]$, $[1,0,0,0,0,2,0,22]$,\\
&  $[0,0,\frac{9}{2},-\frac{9}{2},\frac{17}{2},-\frac{17}{2},4,1]$& $[0,0,0,0,0,0,10,10]$, $[0,0,0,0,0,5,0,20]$, \\
& &  $[n,0,0,0,0,1,1,2n+19]$, $0\leq n\leq 3$\\

$54724^*$ & $[1,2,2,-1,2,-5,8,1]$  & $[5,1,0,0,0,0,0,5]$, $[4,1,1,0,0,0,0,7]$, \\
&  $[\frac{3}{2},\frac{3}{2},\frac{3}{2},-\frac{3}{2},\frac{3}{2},-8,10,0]$& $[3,1,2,0,0,0,0,9]$, $[3,3,0,0,0,0,2,5]$,\\
& & $[0,6,0,0,0,0,0,10]$, $[2,2,2,0,0,0,0,10]$,\\
& & $[1,5,0,0,0,0,1,8]$, $[2,4,0,0,0,0,2,6]$\\

$40871$ & $[3,2,2,-1,1,-3,4,2]$  & $\textbf{[0,0,0,0,0,0,3,27]}$, $[1,0,0,0,0,0,2,28]$, \\
&  $[4,\frac{5}{2},\frac{5}{2},-\frac{5}{2},0,-\frac{13}{2},\frac{13}{2},\frac{5}{2}]$& $[1,1,0,0,0,0,1,30]$, $[1,2,0,0,0,0,0,32]$\\

$40869$ & $[3,2,2,-1,1,-3,4,2]$  & $[0,0,0,0,0,4,3,11]$, $[0,0,0,0,1,2,4,13]$, \\
&  $[4,\frac{5}{2},\frac{5}{2},-\frac{5}{2},0,-\frac{13}{2},\frac{13}{2},\frac{5}{2}]$& $[0,0,0,0,2,0,5,15]$, $[0,0,0,0,0,5,2,10]$\\

$42386$ & $[2,7,1,-4,6,-5,2,1]$  & $[0,3,1,0,0,0,0,17]$, $[0,3,0,0,1,0,0,18]$, \\
&  $[\frac{3}{2},\frac{15}{2},0,-\frac{11}{2},7,-7,2,0]$& $[0,2,2,0,0,0,0,18]$, $[0,2,1,0,1,0,0,19]$,\\
& & $[0,3,1,0,0,0,3,14]$, $[0,4,0,0,0,0,4,12]$,\\
& & $[0,1,3,0,0,0,1,18]$, $[0,0,4,0,0,0,0,20]$\\

$38638$ & $[2,1,1,-3,6,-5,6,2]$  & $\textbf{[0,0,4,0,0,0,0,0]}$, $[0,0,3,1,0,0,0,2]$, \\
&  $[1,1,1,-5,6,-6,7,1]$&  $[0,0,2,2,0,0,0,4]$, $[0,4,2,0,0,0,0,0]$\\

$38573$ & $[1,7,1,-5,6,-3,1,3]$  & $\textbf{[0,0,0,0,0,5,0,0]}$, $[1,0,0,0,0,5,0,2]$, \\
&  $[0,9,0,-6,6,-5,1,1]$& $[0,0,0,0,2,3,0,2]$, $[0,0,0,1,1,3,0,1]$,\\
&                       & $[0,0,0,0,3,2,0,3]$, $[0,0,0,1,2,2,0,2]$ \\

$32571$ & $[3,2,1,-4,5,-2,3,3]$  & $[6,0,0,0,0,0,2,16]$, $[4,1,0,0,0,0,3,16]$, \\
&  $[\frac{5}{2},3,0,-6,6,-5,5,\frac{5}{2}]$& $[2,1,0,0,0,2,3,12]$, $[3,0,0,0,0,3,2,10]$,\\
& & $[1,2,0,0,0,1,4,14]$, $[0,3,0,0,0,0,5,16]$\\

$31085$ & $[1,3,1,-2,3,-2,3,1]$  & $\textbf{[0,0,0,0,0,2,0,28]}$, $[0,0,1,0,0,1,0,30]$, \\
&  $[\frac{5}{2},3,0,-\frac{11}{2},\frac{11}{2},-\frac{11}{2},\frac{11}{2},\frac{5}{2}]$& $[0,0,2,0,0,0,0,32]$\\

$31081$ & $[1,3,1,-2,3,-2,3,1]$  & $[5,0,0,0,0,0,3,17]$, $[2,0,0,0,0,3,3,11]$, \\
&  $[\frac{5}{2},3,0,-\frac{11}{2},\frac{11}{2},-\frac{11}{2},\frac{11}{2},\frac{5}{2}]$& $[1,1,0,0,0,2,4,13]$, $[0,2,0,0,0,1,5,15]$\\

$31080$ & $[1,3,1,-2,3,-2,3,1]$  & $[5,1,0,0,0,0,2,15]$, $[3,1,0,0,0,2,2,11]$, \\
&  $[\frac{5}{2},3,0,-\frac{11}{2},\frac{11}{2},-\frac{11}{2},\frac{11}{2},\frac{5}{2}]$& $[2,2,0,0,0,1,3,13]$, $[1,3,0,0,0,0,4,15]$
\end{tabular}
\label{table-EIX-11101011-part1}
\end{table}

\begin{table}[H]
\centering
\caption{Infinitesimal character $[1,1,1,0,1,0,1,1]$, part II}
\begin{tabular}{lcc}
$\# x$ & $\lambda/\nu$ &  Spin LKTs   \\
\hline
$26299$ & $[1,1,5,-3,1,-1,4,1]$  & $[2,0,0,0,0,4,1,1]$, $[2,1,0,0,0,3,2,3]$, \\
&  $[0,0,7,-5,1,-3,4,1]$& $[1,2,0,0,0,3,2,4]$\\

$24960$ & $[1,1,4,-2,1,-1,3,1]$  & $[0,6,0,0,0,0,0,10]$, $[2,2,2,0,0,0,0,10]$, \\
&  $[1,0,\frac{13}{2},-\frac{9}{2},0,-\frac{5}{2},\frac{9}{2},0]$& $[1,1,2,1,0,0,0,11]$, $[0,1,1,2,0,0,0,13]$\\

$24618$ & $[1,1,4,-2,1,-1,3,1]$  & $\textbf{[0,3,0,0,0,0,2,21]}$, $[1,3,0,0,0,0,1,22]$, \\
&  $[0,1,7,-5,0,-\frac{5}{2},\frac{9}{2},0]$& $[2,2,0,0,0,1,0,24]$, $[0,1,0,0,0,4,0,21]$\\

$20809$ & $[1,3,1,-2,3,-2,4,-1]$  & $\textbf{[1,0,0,0,0,0,0,34]}$, $[0,0,0,0,0,1,0,36]$ \\
&  $[0,\frac{11}{2},0,-\frac{11}{2},\frac{11}{2},-\frac{11}{2},5,1]$&\\

$14188$ & $[1,1,3,-2,1,0,3,1]$  & $[1,0,0,0,0,1,8,8]$, $[0,0,0,0,1,1,7,6]$, \\
&  $[0,0,\frac{9}{2},-\frac{9}{2},0,0,4,1]$& $[0,0,0,0,2,0,7,5]$\\

$15961$ & $[1,6,2,-4,2,-1,3,1]$  & $\textbf{[0,0,0,0,4,0,0,10]}$, $[2,0,0,0,2,2,0,8]$, \\
&  $[1,6,1,-5,1,-1,2,1]$& $[0,1,0,0,3,1,0,8]$, $[1,0,1,0,1,3,0,9]$\\

$10387$ & $[1,1,3,-2,2,-1,1,4]$  & $[0,0,0,0,0,0,1,37]$ \\
&  $[0,0,4,-4,3,-3,0,5]$&\\

$12847$ & $[1,3,2,-2,1,0,2,1]$  & $\textbf{[0,0,0,0,3,0,2,15]}$, $[1,0,0,0,3,0,1,16]$, \\
&  $[1,5,2,-5,0,0,2,1]$& $[0,2,0,0,1,0,4,17]$, $[0,1,0,1,1,0,3,15]$\\

$9400$ & $[1,1,3,-2,3,-2,1,3]$  & $\textbf{[0,0,0,0,0,0,0,38]}$ \\
&  $[0,0,4,-4,4,-4,0,4]$&\\

$9276$ & $[2,0,0,1,0,1,1,1]$  & $\textbf{[1,0,2,0,0,0,0,28]}$, $[2,0,0,1,0,0,0,30]$ \\
&  $[\frac{3}{2},\frac{9}{2},\frac{3}{2},-\frac{9}{2},\frac{3}{2},-\frac{3}{2},0,3]$&\\

$9275$ & $[2,0,0,1,0,1,1,1]$  & $[1,4,0,0,0,0,4,2]$, $[0,3,0,1,0,0,4,3]$ \\
&  $[\frac{3}{2},\frac{9}{2},\frac{3}{2},-\frac{9}{2},\frac{3}{2},-\frac{3}{2},0,3]$&\\

$6233$ & $[2,2,1,-1,1,0,1,2]$  & $\textbf{[0,1,0,0,0,0,10,3]}$, $[0,0,1,0,0,0,10,2]$ \\
&  $[3,3,0,-3,0,-\frac{1}{2},1,\frac{5}{2}]$&\\

$6232$ & $[2,2,1,-1,1,0,1,2]$  & $[4,1,1,0,0,0,0,21]$, $[3,0,1,0,1,0,0,23]$ \\
&  $[3,3,0,-3,0,-\frac{1}{2},1,\frac{5}{2}]$&\\

$6920$ & $[1,1,2,-2,4,-2,1,2]$  & $\textbf{[0,0,0,2,0,0,5,7]}$, $[1,0,0,2,0,0,4,8]$, \\
&  $[1,1,1,-\frac{7}{2},\frac{9}{2},-3,0,\frac{3}{2}]$&$[0,1,0,1,0,1,5,8]$, $[1,0,0,1,1,0,5,6]$\\

$5549$ & $[2,2,1,-2,3,-1,1,2]$  & $\textbf{[0,0,2,0,0,3,0,16]}$, $[1,1,1,0,0,3,0,15]$, \\
&  $[2,2,0,-\frac{7}{2},\frac{7}{2},-\frac{5}{2},1,1]$&$[0,0,1,1,0,2,1,17]$\\

$4980$ & $[2,2,1,-2,3,-1,1,2]$  & $\textbf{[4,0,0,0,1,0,0,25]}$, $[4,0,0,1,0,0,0,24]$, \\
&  $[2,2,0,-\frac{7}{2},\frac{7}{2},-2,0,\frac{3}{2}]$& $[4,1,0,0,0,0,1,26]$
\end{tabular}
\label{table-EIX-11101011-part2}
\end{table}

\begin{table}[H]
\centering
\caption{Infinitesimal character $[1,1,1,0,1,1,0,1]$, part I}
\begin{tabular}{lcc}
$\# x$ & $\lambda/\nu$ &  Spin LKTs   \\
\hline

$52964^*$ & $[1,2,2,-1,2,2,-5,6]$  & $[6,1,0,0,0,0,0,5]$, $[5,2,0,0,0,0,0,6]$, \\
&  $[\frac{3}{2},\frac{3}{2},\frac{3}{2},-\frac{3}{2},\frac{3}{2},2,-10,10]$& $[5,1,1,0,0,0,0,7]$, $[4,2,1,0,0,0,0,8]$,\\
& & $[4,3,0,0,0,0,1,6]$, $[3,3,1,0,0,0,0,9]$, \\
& &$[2,5,0,0,0,0,0,9]$, $[3,4,0,0,0,0,1,7]$, \\
&  &$[1,6,0,0,0,0,0,10]$, $[2,5,0,0,0,0,1,8]$ \\

$38508$ & $[3,1,1,0,1,1,-2,3]$  &  $\textbf{[0,0,0,0,0,0,4,28]}$, $[1,0,0,0,0,0,3,29]$,\\
&  $[4,\frac{5}{2},\frac{5}{2},-\frac{5}{2},0,0,-\frac{13}{2},9]$& $[2,0,0,0,0,0,2,30]$, $[1,1,0,0,0,0,2,31]$,\\
& & $[2,1,0,0,0,0,1,32]$\\

$38506$ & $[3,1,1,0,1,1,-2,3]$  & $[0,0,0,0,0,3,5,13]$, $[0,0,0,0,1,1,6,15]$, \\
&  $[4,\frac{5}{2},\frac{5}{2},-\frac{5}{2},0,0,-\frac{13}{2},9]$& $[0,0,0,0,0,4,4,12]$, $[0,0,0,0,1,2,5,14]$,\\
& &  $[0,0,0,0,0,5,3,11]$\\

$36211$ & $[2,1,2,-3,4,2,-4,6]$  & $\textbf{[1,0,4,0,0,0,0,0]}$, $[0,1,4,0,0,0,0,1]$, \\
&  $[1,1,1,-5,6,1,-7,8]$& $[1,0,3,1,0,0,0,2]$, $[0,1,3,1,0,0,0,3]$,\\
& & $[0,3,3,0,0,0,0,1]$\\

$28625$ & $[2,2,1,-2,3,1,-2,4]$  & $\textbf{[0,0,0,0,0,2,1,29]}$, $[1,0,0,0,0,2,0,30]$, \\
&  $[\frac{5}{2},3,0,-\frac{11}{2},\frac{11}{2},0,-\frac{11}{2},8]$& $[0,0,1,0,0,1,1,31]$, $[1,0,1,0,0,1,0,32]$\\

$28621$ & $[2,2,1,-2,3,1,-2,4]$  & $[5,0,0,0,0,0,4,18]$, $[2,0,0,0,0,2,5,13]$, \\
&  $[\frac{5}{2},3,0,-\frac{11}{2},\frac{11}{2},0,-\frac{11}{2},8]$& $[1,1,0,0,0,1,6,15]$, $[2,0,0,0,0,3,4,12]$,\\
& & $[1,1,0,0,0,2,5,14]$\\

$28620$ & $[2,2,1,-2,3,1,-2,4]$  & $[5,1,0,0,0,0,3,14]$, $[4,1,0,0,0,1,3,12]$, \\
&  $[\frac{5}{2},3,0,-\frac{11}{2},\frac{11}{2},0,-\frac{11}{2},8]$& $[3,2,0,0,0,0,4,14]$, $[3,2,0,0,0,1,3,13]$,\\
& & $[2,3,0,0,0,0,4,15]$\\

$23882$ & $[1,1,4,-2,1,1,-1,4]$  & $\textbf{[2,0,0,0,0,5,0,0]}$, $[3,0,0,0,0,4,1,1]$, \\
&  $[0,0,7,-5,1,1,-4,5]$& $[2,1,0,0,0,4,1,2]$, $[1,2,0,0,0,4,1,3]$\\

$22535$ & $[1,1,4,-2,1,2,-2,3]$  & $[3,3,1,0,0,0,0,11]$, $[1,6,0,0,0,0,0,10]$, \\
&  $[1,0,\frac{13}{2},-\frac{9}{2},0,2,-\frac{9}{2},\frac{9}{2}]$& $[2,2,1,1,0,0,0,12]$, $[1,2,0,2,0,0,0,14]$,\\
& & $[1,1,1,2,0,0,0,13]$\\

$22221$ & $[1,1,4,-2,1,2,-2,3]$  & $\textbf{[0,3,0,0,0,0,3,22]}$, $[1,3,0,0,0,0,2,23]$,\\
&  $[0,1,7,-5,0,2,-\frac{9}{2},\frac{9}{2}]$&  $[2,3,0,0,0,0,1,24]$, $[0,2,0,0,0,3,1,21]$,\\
& & $[2,2,0,0,0,1,1,25]$
\end{tabular}
\label{table-EIX-11101101-I}
\end{table}

\begin{table}[H]
\centering
\caption{Infinitesimal character $[1,1,1,0,1,1,0,1]$, part II}
\begin{tabular}{lcc}
$\# x$ & $\lambda/\nu$ &  Spin LKTs   \\
\hline
$12214$ & $[1,1,3,-2,1,2,-1,4]$  &  $[1,0,0,0,0,0,1,37]$\\
&  $[0,0,\frac{9}{2},-\frac{9}{2},0,4,-4,5]$& \\

$12211$ & $[1,1,3,-2,1,2,-1,4]$  & $\textbf{[0,0,0,0,0,2,8,8]}$, $[1,0,0,0,0,1,9,9]$, \\
&  $[0,0,\frac{9}{2},-\frac{9}{2},0,4,-4,5]$& $[0,0,0,0,1,1,8,7]$, $[0,0,0,0,2,0,8,6]$\\

$13888$ & $[2,5,1,-3,2,1,-1,3]$  & $\textbf{[1,0,0,0,4,0,0,10]}$, $[2,0,0,0,3,1,0,9]$, \\
&  $[1,6,1,-5,1,1,-2,3]$& $[0,1,0,0,4,0,0,9]$, $[1,1,0,0,3,1,0,8]$,\\
& & $[1,0,1,0,2,2,0,10]$\\

$11005$ & $[2,0,0,1,1,0,1,1]$  & $[1,4,0,0,0,0,5,1]$, $[0,3,0,1,0,0,5,2]$ \\
&  $[1,5,2,-5,0,2,-2,3]$&\\

$11004$ & $[2,0,0,1,1,0,1,1]$  & $\textbf{[0,0,3,0,0,0,0,28]}$, $[1,0,2,0,0,0,1,29]$, \\
&  $[1,5,2,-5,0,2,-2,3]$& $[1,0,1,1,0,0,0,30]$\\

$11002$ & $[2,0,0,1,1,0,1,1]$  & $\textbf{[0,0,0,0,3,0,3,16]}$, $[0,1,0,0,2,0,4,17]$, \\
&  $[1,5,2,-5,0,2,-2,3]$& $[1,0,0,0,3,0,2,17]$, $[0,0,0,1,2,0,3,15]$,\\
& & $[0,1,0,1,1,0,4,16]$\\

$5206$ & $[2,2,1,-1,1,1,0,2]$  & $\textbf{[0,1,0,0,0,0,11,4]}$, $[0,0,1,0,0,0,11,3]$ \\
&  $[3,3,0,-3,0,0,0,3]$& \\

$5196$ & $[2,2,1,-1,1,1,0,2]$  & $\textbf{[4,0,2,0,0,0,0,22]}$, $[5,1,1,0,0,0,0,21]$, \\
&  $[3,3,0,-3,0,0,0,3]$& $[4,0,1,0,1,0,0,23]$
\end{tabular}
\label{table-EIX-11101101-II}
\end{table}

\begin{table}[H]
\centering
\caption{Infinitesimal character $[1,1,1,0,1,1,1,1]$}
\begin{tabular}{lcc}
$\# x$ & $\lambda/\nu$ &  Spin LKTs   \\
\hline
$67078_{\clubsuit}^*$ & $[3,2,2,-1,1,1,2,1]$  & $[0,0,0,0,0,0,0,8]+n\beta$, $1\leq n\leq 10$ \\
&  $[4,\frac{5}{2},\frac{5}{2},-\frac{5}{2},0,0,\frac{5}{2},\frac{5}{2}]$& \\

$59541$ & $[1,1,1,-2,4,1,1,1]$  & $\textbf{[7,0,0,0,0,0,0,0]}$,  \\
&  $[1,0,1,-9,11,0,2,0]$& $[7,0,0,0,0,0,0,0]+n\beta$, $1\leq n\leq 4$\\

$52495$ & $[2,1,1,-2,3,1,2,1]$  &  $\textbf{[0,0,0,0,0,0,6,22]}$, $[0,0,0,0,0,0,10,18]$,\\
&  $[4,0,0,-\frac{17}{2},\frac{17}{2},0,4,1]$& $[0,0,0,0,0,n,6-n,22+n]$, $1\leq n\leq 4$\\

$35211$ & $[2,2,5,-4,2,1,1,1]$  & $\textbf{[0,6,0,0,0,0,0,18]}$, $[0,6,0,0,0,0,1,17]$, \\
&  $[\frac{3}{2},\frac{3}{2},9,-9,2,0,2,0]$& $[0,5,1,0,0,0,0,19]$, $[0,6,0,0,0,0,2,16]$,\\
& & $[0,4,2,0,0,0,0,20]$\\

$31203$ & $[1,1,3,-2,2,1,1,1]$  & $\textbf{[0,0,0,0,0,7,0,0]}$, $[0,0,0,0,1,6,0,1]$, \\
&  $[0,0,9,-9,3,1,1,1]$& $[0,0,0,0,2,5,0,2]$\\

$14191$ & $[1,1,2,-1,1,1,2,1]$  & $\textbf{[3,0,0,0,0,0,0,38]}$, $[2,1,0,0,0,0,0,39]$, \\
&  $[0,0,\frac{11}{2},-\frac{11}{2},0,0,5,1]$& $[1,2,0,0,0,0,0,40]$\\

$18132$ & $[1,4,1,-2,1,2,1,1]$  & $\textbf{[0,0,0,0,5,0,0,13]}$, $[1,0,0,0,4,1,0,14]$, \\
&  $[1,\frac{15}{2},1,-6,0,\frac{3}{2},1,1]$& $[0,0,0,1,4,0,0,12]$\\

$12850$ & $[1,3,2,-2,1,1,2,1]$  & $\textbf{[0,4,0,0,0,0,8,0]}$, $[0,3,1,0,0,0,8,1]$ \\
&  $[1,6,\frac{5}{2},-6,0,0,\frac{5}{2},1]$& \\

$12849$ & $[1,3,2,-2,1,1,2,1]$  & $\textbf{[0,0,4,0,0,0,0,28]}$, $[0,0,4,0,0,0,1,27]$, \\
&  $[1,6,\frac{5}{2},-6,0,0,\frac{5}{2},1]$& $[0,0,3,0,1,0,0,29]$\\

$5214$ & $[2,2,1,-1,1,1,1,2]$  & $\textbf{[0,1,0,0,0,0,0,43]}$, $[0,0,0,0,0,1,0,44]$, \\
&  $[\frac{7}{2},\frac{7}{2},0,-\frac{7}{2},0,0,0,\frac{7}{2}]$&
\end{tabular}
\label{table-EIX-11101111}
\end{table}

\begin{table}[H]
\centering
\caption{Infinitesimal character $[1,1,1,1,0,1,0,1]$}
\begin{tabular}{lcc}
$\# x$ & $\lambda/\nu$ &  Spin LKTs   \\
\hline

$33745$ & $[2,1,2,1,-3,5,-4,6]$  & $\textbf{[0,0,5,0,0,0,0,0]}$, $[0,0,4,1,0,0,0,2]$, \\
&  $[1,1,1,1,-6,7,-7,8]$& $[0,2,4,0,0,0,0,0]$\\

$27627$ & $[3,2,1,1,-4,5,-2,5]$  & $[6,0,0,0,0,0,4,16]$, $[4,1,0,0,0,0,5,16]$, \\
&  $[\frac{5}{2},3,0,0,-6,6,-5,\frac{15}{2}]$& $[3,1,0,0,0,1,5,14]$, $[4,0,0,0,0,2,4,12]$,\\
& & $[2,2,0,0,0,0,6,16]$\\

$26187$ & $[1,2,1,1,-1,2,-1,2]$  & $\textbf{[0,0,0,0,0,3,0,30]}$, $[0,0,1,0,0,2,0,32]$ \\
&  $[\frac{5}{2},3,0,0,-\frac{11}{2},\frac{11}{2},-\frac{11}{2},8]$& \\

$26183$ & $[1,2,1,1,-1,2,-1,2]$  & $[5,0,0,0,0,0,5,17]$, $[3,0,0,0,0,2,5,13]$, \\
&  $[\frac{5}{2},3,0,0,-\frac{11}{2},\frac{11}{2},-\frac{11}{2},8]$& $[2,1,0,0,0,1,6,15]$\\

$26182$ & $[1,2,1,1,-1,2,-1,2]$  & $[5,1,0,0,0,0,4,15]$, $[4,1,0,0,0,1,4,13]$, \\
&  $[\frac{5}{2},3,0,0,-\frac{11}{2},\frac{11}{2},-\frac{11}{2},8]$& $[3,2,0,0,0,0,5,15]$\\

$11597$ & $[1,2,1,1,-2,4,-2,2]$  & $\textbf{[0,0,3,0,0,0,6,10]}$, $[1,0,3,0,0,0,5,11]$, \\
&  $[0,\frac{3}{2},2,0,-\frac{9}{2},6,-5,1]$& $[1,1,2,0,0,0,6,9]$, $[0,0,2,0,1,0,6,11]$\\

$8329$ & $[3,1,2,1,-3,4,-2,2]$  & $\textbf{[4,0,0,0,0,3,0,22]}$, $[4,0,0,0,1,2,0,21]$, \\
&  $[1,0,3,0,-\frac{9}{2},\frac{9}{2},-\frac{7}{2},1]$& $[3,0,1,0,0,2,1,23]$
\end{tabular}
\label{table-EIX-11110101}
\end{table}

\begin{table}[H]
\centering
\caption{Infinitesimal character $[1,1,1,1,0,1,1,1]$}
\begin{tabular}{lcc}
$\# x$ & $\lambda/\nu$ &  Spin LKTs   \\
\hline
$58238$ & $[1,1,1,2,-4,5,2,1]$  & $\textbf{[8,0,0,0,0,0,0,0]}$, \\
&  $[1,0,1,2,-11,11,2,0]$& $[8,0,0,0,0,0,n,n]$, $1\leq n\leq 3$\\

$50615$ & $[2,1,1,1,-2,3,2,1]$  &  $\textbf{[0,0,0,0,0,0,7,23]}$, $[0,0,0,0,0,0,10,20]$,\\
&  $[4,0,0,0,-\frac{17}{2},\frac{17}{2},4,1]$&  $[0,0,0,0,0,n,7-n,23+n]$, $1\leq n\leq 3$ \\

\end{tabular}
\label{table-EIX-11110111}
\end{table}

\begin{table}[H]
\centering
\caption{Infinitesimal character $[1,1,1,1,1,0,1,0]$}
\begin{tabular}{lcc}
$\# x$ & $\lambda/\nu$ &  Spin LKTs   \\
\hline
$22704$ & $[2,1,1,1,1,-2,4,-1]$  & $[6,0,0,0,0,0,6,16]$, $[5,0,0,0,0,1,6,14]$, \\
&  $[\frac{5}{2},3,0,0,0,-6,\frac{17}{2},-\frac{15}{2}]$& $[4,1,0,0,0,0,7,16]$
\end{tabular}
\label{table-EIX-11111010}
\end{table}

\begin{table}[H]
\centering
\caption{Infinitesimal character $[1,1,1,1,1,0,1,1]$}
\begin{tabular}{lcc}
$\# x$ & $\lambda/\nu$ &  Spin LKTs   \\
\hline
$56790$ & $[1,1,1,1,1,-1,3,1]$  & $\textbf{[9,0,0,0,0,0,0,0]}$,  \\
&  $[1,0,1,2,0,-11,13,0]$& $[9,0,0,0,0,0,1,1]$, $[9,0,0,0,0,0,2,2]$ \\

$48602$ & $[1,1,1,1,1,-1,3,1]$  & $\textbf{[0,0,0,0,0,0,8,24]}$, $[0,0,0,0,0,1,7,25]$, \\
&  $[4,0,0,0,0,-\frac{17}{2},\frac{25}{2},1]$& $[0,0,0,0,0,0,10,22]$, $[0,0,0,0,0,2,6,26]$
\end{tabular}
\label{table-EIX-11111011}
\end{table}

\begin{table}[H]
\centering
\caption{Infinitesimal character $[1,1,1,1,1,1,0,1]$}
\begin{tabular}{lcc}
$\# x$ & $\lambda/\nu$ &  Spin LKTs   \\
\hline
$55211$ & $[1,1,1,1,1,2,-2,3]$  & $\textbf{[10,0,0,0,0,0,0,0]}$, $[10,0,0,0,0,0,1,1]$ \\
&  $[1,0,1,2,0,2,-13,13]$& \\

$46497$ & $[1,1,1,1,1,1,-1,4]$  & $\textbf{[0,0,0,0,0,0,9,25]}$, $[0,0,0,0,0,0,10,24]$, \\
&  $[4,0,0,0,0,4,-\frac{25}{2},\frac{27}{2}]$& $[0,0,0,0,0,1,8,26]$
\end{tabular}
\label{table-EIX-11111101}
\end{table}

\begin{table}[H]
\centering
\caption{Infinitesimal character $[1, 1, 1, 1, 1, 1, 1, 1]$}
\begin{tabular}{lccc}
$\# x$ & $\lambda$ & $\nu$ &  Spin LKT   \\
\hline
$67109_{\clubsuit}$ & $[1,1,1,1,1,1,1,1]$  & $[4,0,0,0,0,4,1,1]$ & $\textbf{[0,0,0,0,0,0,0,0]}$   	
\end{tabular}
\label{table-EIX-11111111}
\end{table}

\centerline{\scshape Acknowledgements}
 Luan thanks the support from both School of Mathematics and Informatization Office, Shandong University. We are all deeply grateful to the \texttt{atlas} mathematicians.

\centerline{\scshape Funding}
Dong is supported by the National Natural Science Foundation of China (grant 12171344). Luan is supported by the Shandong Provincial Natural Science Foundation under Grant ZR2022QA056.

\end{document}